\documentclass[11pt,letterpaper]{article}
\usepackage{geometry,float}
\usepackage{amssymb}
\usepackage{amstext}
\usepackage{graphics}
\usepackage{graphicx}
\usepackage{amsthm}
\usepackage{amsmath}
\usepackage{mathtools}
\usepackage{color}
\usepackage{graphics}
\usepackage{caption}
\usepackage{rotating}
\usepackage{multirow}
\usepackage[colorlinks=true, allcolors=blue]{hyperref}

\usepackage{natbib}

\bibpunct{(}{)}{;}{a}{}{,}
\graphicspath{{figures/}}

\newcommand*{\diff}[2]{\frac{\partial #1}{\partial #2}} 
\AtBeginDocument{\renewcommand{\d}{\mathop{}\!\mathrm{d}}}
\makeatletter
\newsavebox{\mybox}\newsavebox{\mysim}
\newcommand{\distas}[1]{%
  \savebox{\mybox}{\hbox{\kern3pt$\scriptstyle#1$\kern3pt}}%
  \savebox{\mysim}{\hbox{$\sim$}}%
  \mathbin{\overset{#1}{\kern\z@\resizebox{\wd\mybox}{\ht\mysim}{$\sim$}}}%
}
\makeatother
\DeclareMathOperator*{\argmin}{arg\,min}

\newcommand{\toinf}{\rightarrow \infty}
\newcommand{\tozero}{\rightarrow 0}

\newcommand{\sumjinf}{\sum_{j=1}^\infty}

\newcommand*{\inv}{^{-1}}

\newcommand{\bbR}{\mathbb{R}}

\newcommand{\bbN}{\mathbb{N}}

\newcommand{\Var}{\,\text{Var}}

\newcommand{\Cov}{\,\text{Cov}}

\newcommand{\one}{\,\mathbf{1}}

\renewcommand{\exp}{\,\text{exp}}
\newcommand{\Supp}{\text{Supp}}

\newcommand{\hbeta}{\hat{\beta}}

\newcommand{\hmu}{\hat{\mu}}

\newcommand{\hx}{\hat{x}}
\newcommand{\hf}{\hat{f}}
\newcommand{\hg}{\hat{g}}

\newcommand{\hE}{\hat{E}}

\newcommand{\hG}{\hat{G}}

\newcommand{\hQ}{\hat{Q}}
\newcommand{\hxi}{\hat{\xi}}
\newcommand{\hpi}{\hat{\pi}}

\newcommand{\hlambda}{\hat{\lambda}}

\newcommand{\hphi}{\hat{\phi}}

\newcommand{\barf}{\bar{f}}
\newcommand{\barg}{\bar{g}}

\newcommand{\tf}{\tilde{f}}

\newcommand{\tQ}{\tilde{Q}}
\newcommand{\tphi}{\tilde{\phi}}

\newcommand{\tx}{\tilde{x}}
\newcommand{\tX}{\tilde{X}}

\newcommand{\txi}{\tilde{\xi}}

\newcommand*{\cT}{\mathcal{T}}
\newcommand*{\cS}{\mathcal{S}}
\newcommand*{\cX}{\mathcal{X}}
\newcommand*{\cU}{\mathcal{U}}

\usepackage{amsthm}
\theoremstyle{compact.definition}

\newtheorem{Theorem}{Theorem}
\newtheorem{Lemma}{Lemma}

\theoremstyle{compact.remark}

\newcommand{\onetoinf}{_{j = 1}^\infty}
\newcommand{\onetoJ}{_{j = 1}^J}
\newcommand*{\oneovernk}{\frac{1}{n_k}}

\newcommand*{\sumink}{\sum_{i=1}^{n_k}}

\newcommand{\CE}{common eigenfunction}
\newcommand{\half}{\frac{1}{2}}

\newcommand*{\lbarM}{\underline{M}}
\newcommand*{\barM}{\overline{M}}
\newcommand{\Xk}{X^{[k]}}
\newcommand{\Xktwo}{\diff{^2}{t^2}\Xk}
\newcommand{\tXik}{\tX^{[k]}_i}

\newcommand{\tXki}{\tXik}
\newcommand{\Xik}{X^{[k]}_i}

\newcommand{\Xki}{\Xik}

\newcommand{\tGk}{\tilde{G}_k}
\newcommand{\tG}{\tilde{G}}
\newcommand{\tDelta}{\tilde{\Delta}}

\newcommand{\bea}{\begin{eqnarray*}}
\newcommand{\eea}{\end{eqnarray*}}
\newcommand{\be}{\begin{eqnarray}}
\newcommand{\ee}{\end{eqnarray}}
\newcommand{\ed}{\end{document}}
\newcommand{\no}{\noindent}

\newcommand{\btab}{\begin{tabular}}
\newcommand{\etab}{\end{tabular}}
\newcommand{\bc}{\begin{center}}
\newcommand{\ec}{\end{center}}

\newcommand{\bi}{\begin{itemize}}
\newcommand{\ei}{\end{itemize}}
\newcommand{\bfi}{\begin{figure}}
\newcommand{\efi}{\end{figure}}
\newcommand{\ben}{\begin{enumerate}}
\newcommand{\een}{\end{enumerate}}
\newcommand{\bdes}{\begin{description}}
\newcommand{\edes}{\end{description}}
\newcommand{\bay}{\begin{array}}
\newcommand{\eay}{\end{array}}

\def\bco{\iffalse}
\def\ci{\cite}
\def\cp{\citep}

\newcommand{\blu}[1]{\textcolor{blue}{#1}}

\definecolor{grey}{rgb}{0.57, 0.64, 0.69}

\def\i1{_{i1}}

\def\cI{{\mathcal I}}

\def\s2{\sum_{k=1}^\infty\lambda_k\phi_k(t) \phi_k(t)}
\def\s1{\sum_{k=1}^\infty\lambda_k\phi^{(1)}_k(s) \phi_k(t)}

\def\pt{\phi_k(t)}
\def\ps{\phi_k(s)}
\def\pt1{\phi^{(1)}_k(t)}
\def\ps1{\phi^{(1)}_k(s)}

\def\singlespace{\def\baselinestretch{1}\@normalsize}

\newcommand{\tight}{\renewcommand{\baselinestretch}{1}\normalsize}

\newcommand{\double}{\renewcommand{\baselinestretch}{1.5}\normalsize}

\double
  
\begin{document}

\thispagestyle{empty}  \bc {\bf \sc \Large Optimal Bayes Classifiers for Functional Data and Density Ratios} \vspace{0.25in}\ec

\vspace*{0.1in} \centerline{\today} \vspace*{0.1in}

\vspace*{0.1in} \centerline{Short title: Functional Bayes Classifiers}
\vspace*{0.1in}

\begin{center}

Xiongtao Dai\footnote{Corresponding author}\\

Department of Statistics \\ %
University of California, Davis \\ %
Davis, CA 95616 U.S.A. \\ %
Email: dai@ucdavis.edu

\vspace*{0.3in} 
Hans-Georg M\"uller\footnote{Research supported by NSF grants DMS-1228369 and DMS-1407852}

Department of Statistics \\ %
University of California, Davis \\ %
Davis, CA 95616 U.S.A. \\ %
Email: hgmueller@ucdavis.edu \\

\vspace*{0.3in} Fang Yao\footnote{Research supported by the Natural Sciences and Engineering Research Council of Canada}\\
Department of Statistics\\
University of Toronto\\
100 St. George Street\\
Toronto, ON M5S 3G3\\
Canada\\
Email: fyao@utstat.toronto.edu\\
\end{center}\vspace{.5in}

\thispagestyle{empty}
\newpage \pagenumbering{arabic}

\vspace{0.05in} \thispagestyle{empty} \bc {\bf ABSTRACT} \ec

\no Bayes classifiers for functional data pose a challenge.  This is because probability density functions do not exist for functional data. As a consequence, the classical Bayes classifier using density quotients needs to be modified. We propose to use  density ratios of projections on a sequence of eigenfunctions that are common to the groups to be classified. The density ratios can then be factored into density ratios of individual functional principal components whence the classification problem is reduced to a sequence of nonparametric one-dimensional density estimates. This is an extension to functional data of some of the very earliest nonparametric Bayes classifiers that were based on simple density ratios in the one-dimensional case. By means of the factorization of the density quotients the curse of dimensionality that would otherwise severely affect Bayes classifiers for functional data can be avoided. We demonstrate that in the case of Gaussian functional data, the proposed functional Bayes classifier reduces to a functional version of the classical quadratic discriminant.  A study of the asymptotic behavior of the proposed classifiers in the large sample limit shows that under certain conditions  the misclassification rate converges to zero, a phenomenon that has been referred to as  ``perfect classification''.
%A common functional principal component (CFPC) assumption and some order condition for the eigenvalues are essential for the analysis of the eigenfunction. 
The proposed classifiers also perform favorably in finite sample applications, as we demonstrate in comparisons with other functional classifiers in simulations and various data applications, including wine spectral data, functional magnetic resonance imaging (fMRI) data for attention deficit hyperactivity disorder (ADHD) patients, and yeast gene expression data. 

\vspace{0.1in}

\no {\it Key words and phrases:} \quad Common functional principal component, density estimation, functional classification, Gaussian process, quadratic discriminant analysis.\\

\no {\bf AMS Subject Classification:} 62M20, 62C10, 62H30 

\thispagestyle{empty} \vfill

\section{Introduction}  
In classification of functional data,  predictors may be viewed as  random trajectories and 
responses are indicators for two or more categories. The goal of functional classification is to assign a group label to each predictor function, i.e., to predict the group label for each of the observed random curves. Functional classification is a rich topic with broad applications in many areas of commerce, medicine and the sciences, and with important applications in pattern recognition, chemometrics and genetics \cp{song:08, zhu:10, zhu:12, fran:12, coff:14}.   Within the functional data analysis (FDA)  framework \cp{rams:05}, each observation is viewed as a smooth random curve on a compact domain. Functional classification also has been recently extended to the related  task of classifying longitudinal data 
\cp{wu:13, wang:14} and has close connections with functional clustering \cp{chio:08}. 

There is a rich body of papers on functional classification, using a vast array of methods,  for example distance-based classifiers \cp{ferr:03:3,alon:12}, $k$-nearest neighbor classifiers  \cp{biau:05, cero:06, biau:10}, Bayesian methods \cp{wang:07,yang:14}, logistic regression \cp{arak:09}, or partial least squares \cp{pred:05,pred:07}, 
%and transformation invariant method by \cite{glen:03}.

It is well known that  Bayes classifiers based on density quotients are  optimal classifiers in the sense of minimizing misclassification rates \citep[see for example][]{bick:00}. In the one-dimensional case, this provided one of the core 
motivations for nonparametric density estimation \cp{fix:51,rose:56,parz:62,wegm:72} but for higher-dimensional cases 
an unrestricted nonparametric approach is subject to the curse of dimensionality \cp{scot:15} and this leads to very slow rates of convergence for estimating the nonparametric densities for dimensions larger than three or four. This  renders the resulting
classifiers practically worthless. The situation is exacerbated in the case of functional predictors, which are infinite-dimensional and therefore associated with a most severe curse of dimensionality. This curse of dimensionality is caused by the small ball problem in function space, meaning the expected number of functions falling into balls with small radius is  vanishingly small, 
which implies that densities do not even exist in most cases \cp{li:99,hall:10}. 

As a consequence, in order to define a
Bayes classifier through density quotients with reasonably good estimation properties, one needs to invoke sensible restrictions. These could be  for example restrictions of the class of predictor processes, an approach that has been  adopted  in  \cite{hall:12:2},  who consider two Gaussian populations with equal covariance using a functional linear discriminant,  which is analogous to the linear discriminant. This is the Bayes classifier in the analogous multivariate case. \cite{gale:15}  propose a functional quadratic method for discriminating two general Gaussian populations, making use of a suitably defined  Mahalanobis distance for functional data.
%However, the inverse covariance operator in the functional case is unbounded so the Mahalanobis distance needs a finite truncation.

In contrast to these approaches, 
we aim here at the construction of a  nonparametric Bayes classifier for functional data. To achieve  this, we project the observations onto an orthonormal basis that is common to the two populations, and construct density ratios through products of the density ratios of the projection scores. The densities themselves are nonparametrically estimated, which is feasible as they are one-dimensional. 
We also provide an alternative implementation of the proposed nonparametric functional Bayes classifier through nonparametric regression. This second implementation of functional Bayes classifiers sometimes works even better than the direct approach through density quotients in finite sample applications.

We obtain conditions for the asymptotic equivalence of the proposed functional nonparametric  Bayes classifiers  and their estimated versions,  and also for asymptotic perfect classification when using our classifiers. The term ``perfect classification" was introduced  in \cite{hall:12:2} to denote conditions where the  misclassification rate converges to zero, as the sample size increases, and we use it in the same sense here. 
Perfect classification in the Gaussian case  requires that there are certain differences between the mean or covariance functions, while such differences are not a prerequisite for the nonparametric approach to succeed. In the special case of Gaussian functional predictors, the proposed classifiers simplify to those considered in \ci{hall:13:1}. Additionally, we extend our theoretical results to cover the practically important situation where the functional data are not fully observed, but rather are observed as noisy measurements that are made on a dense grid.
 
In \autoref{s:model}, we introduce the proposed Bayes classifiers and their estimates are discussed in  \autoref{s:est}. We do not require knowledge about the type of underlying processes that generate the functional data. One difficulty for the theoretical analysis that will be addressed in \autoref{s:theory} is that the 
projection scores themselves are not available but rather have to be estimated from the data. Practical implementation, simulation results and applications to various functional data examples are discussed in  \autoref{s:imp}, \autoref{s:sim} and \autoref{s:data}, respectively.   We 
demonstrate that the finite sample performance of the proposed classifiers in simulation studies  and also for three data sets is excellent, specifically in comparison to the  functional linear \citep{hall:12:2}, functional quadratic \citep{gale:15}, and functional logistic regression methods \citep{jame:02,mull:05:1, mull:06:3,esca:07}.

\section{Functional Bayes Classifiers} \label{s:model}
We consider the situation where the observed data come from a common distribution $(X, Y)$, where $X$ is a fully observed square integrable random function in $L^2(\cT)$, $\cT$ is a compact interval, and $Y \in \{0, 1\}$ is a group label. Assume $X$ is distributed as $X^{[k]}$ if $X$ is from population $\Pi_k$, $k = 0, 1$, that is, $X^{[k]}$ is the conditional distribution of $X$ given $Y = k$. Also let $\pi_k = P(Y=k)$ be the prior probability that an observation falls into $\Pi_k$.  Our aim is to infer the group label $Y$ of a new observation $X$. 

The optimal Bayes classification rule that minimizes misclassification error classifies an observation $X=x$ to $\Pi_1$ if
\begin{equation}
\label{eq:ratio1}
Q(x) = \frac{P(Y=1|X=x)}{P(Y=0|X=x)} > 1,
\end{equation} where we denote  realized functional observations by $x$  and random predictor functions by $X$. If  the conditional densities of the functional observations $X$ exist, where conditioning is on  the respective group label, we denote them  as $g_0$ and $g_1$ when conditioning on group 0 or 1.  Then the Bayes theorem implies 
\begin{equation}
\label{eq:ratio2}
Q(x) = \frac{\pi_1 g_1(x)}{\pi_0 g_0(x)}.
\end{equation}
However, the densities for functional data do not usually exist \citep[see][]{hall:10}. To overcome this difficulty, we consider a sequence of approximations to the functional observations, where the number of components is increasing, and then use the density ratios~\eqref{eq:ratio2}. 

  Our approach is to first
represent $x$ and the random $X$ by projecting onto an orthogonal basis $\{\psi_j \}\onetoinf$. This leads to  the projection scores $\{x_j \}\onetoinf$ and $\{\xi_j \}\onetoinf$, where $x_j = \int_{\cT} x(t)\psi_j(t) \d t$ and  $\xi_j = \int_{\cT} X(t)\psi_j(t) \d t$, $j = 1,\,2,\,\dots$. As noted in \cite{hall:01:3}, when comparing the conditional probabilities, it is sensible to project the data from both groups onto the same basis. Our goal is to approximate the conditional probabilities $P(Y=k|X=x)$ by $P(Y=k|\text{the first $J$ scores of $x$})$, where $J\toinf$. Then by Bayes theorem,
\begin{align}
Q(x) & \approx %\lim_{J\rightarrow \infty} 
	\frac{ P(Y=1|\text{the first $J$ scores of $x$})}{P(Y=0|\text{the first $J$ scores of $x$})} \nonumber\\
& = %\lim_{J\rightarrow \infty} 
	\frac{\pi_1 f_1(x_{1}, \dots,x_{J})}{\pi_0 f_0(x_{1},  \dots,x_{J}) }, \label{eq:q}
\end{align}
where $f_1$ and $f_0$ are the conditional densities for the first $J$ random projection scores $\xi_{1}, \dots, \xi_{J}$. 

Estimating the joint density of $(\xi_{1}, \dots,\xi_{J})$ is impractical and subject to the curse of dimensionality when $J$ is large, so it is sensible to introduce reasonable conditions that simplify \eqref{eq:q}. A first simplification is to assume the auto-covariances of the stochastic processes that generate the observed data have the same ordered eigenfunctions for both populations. 
Specifically, write $G_k(s, t) = \Cov(X^{[k]}(s), X^{[k]}(t))$, and define the covariance operators of $G_k(s, t)$ as
\[
G_k: L^2(\mathcal{T}) \rightarrow L^2(\cT),\quad G_k(f) = \int_\cT G_k(s, t) f(s) \d s.
\]
Assuming $G_k(s, t)$ is continuous, by Mercer's theorem 
\begin{gather}
G_k(s, t) = \sum_{j=1}^\infty \lambda_{jk} \phi_{jk}(s) \phi_{jk}(t),
\end{gather}
where $\lambda_{1k}\ge \lambda_{2k} \ge \dots \ge 0$ are the eigenvalues of $G_k$ and $\phi_{jk}$ are the corresponding orthonormal eigenfunctions, $j=1, 2,\dots$, and $\sum_{j=1}^\infty \lambda_{jk} < \infty$, for $k = 0, 1$. The \CE{} condition then is  $\phi_{j0} = \phi_{j1} \eqqcolon \phi_{j}$, for $j = 1,\, 2,\, \dots$ \citep{flur:84, benk:09,boen:10,coff:11}. We note that this assumption can be weaken to that the two populations share the same set of eigenfunctions, not necessarily with the same order. In this case, we reorder the eigenfunctions and eigenvalues such that $\phi_{j0} = \phi_{j1} = \phi_{j}$ holds, but $\lambda_{jk}$ are not necessarily in descending order for $k=0, 1$.

Choosing the projection directions $\{\psi_j \}_{j=1}^\infty$ as the shared eigenfunctions $\{\phi_j \}_{j=1}^\infty$, one has $\Cov(\xi_{j}, \xi_{l}) = 0$ if $j \ne l$.  We note that in general the score $\xi_j$ is not the functional principal component (FPC) $\int_{\cT} (X(t) - \mu_k(t)) \phi_j(t) \d t$. 

A second  simplification is that  we assume that  the projection scores are independent under both populations. Then the densities in  (\ref{eq:q}) factor and the criterion function can be rewritten by taking logarithm as 
\begin{equation} \label{eq:fr_trunc}
Q_J(x) = \log\left( \frac{\pi_1}{\pi_0} \right) +  \sum_{j=1}^J\log\left( \frac{f_{j1}(x_j)}{f_{j0}(x_j)} \right),
\end{equation}
where $f_{jk}$ is the density of the $j$th score under $\Pi_k$. We classify into $\Pi_1$ if and only if $Q_J(x) > 0.$  Due to the zero divided by zero problem, \eqref{eq:fr_trunc} is defined only on a set $\cX$ with $P(X \in \cX) = 1$. Our theoretical arguments in the following are restricted to this set. For the asymptotic analysis we will consider the case where $J = J(n) \rightarrow \infty$ as $n \rightarrow \infty$.

When predictor processes $X$ are Gaussian for $k = 0, 1$, the projection scores $\xi_j$ are independent and one may substitute Gaussian densities for the densities $f_{jk}$ in \eqref{eq:fr_trunc}. 
Define the $j$th projection of the mean function $\mu_k(t)= E(X^{[k]}(t))$ of $\Pi_k$ as 
\[
\mu_{jk} = E(\xi_j|Y = k) = \int_{\cT} \mu_k(t) \phi_j(t) \d t.
\]
Then in this special case of our more general nonparametric approach, one obtains the simplified version 
\begin{align}
Q_J^G(x) & = \log\left(\frac{\pi_1}{\pi_0}\right) + \sum_{j=1}^J \log\left( \frac{(2\pi \lambda_{j1})^{-1/2} \exp(-\frac{(x_j - \mu_{j1})^2}{2\lambda_{j1}}) }{ (2\pi \lambda_{j0})^{-1/2} \exp(-\frac{(x_j - \mu_{j0})^2}{2\lambda_{j0}}) } \right) \nonumber\\
& = \log\left(\frac{\pi_1}{\pi_0}\right) + \half\sum_{j=1}^J \left[(\log\lambda_{j0} - \log\lambda_{j1}) - \left( \frac{1}{\lambda_{j1}}(x_j - \mu_{j1})^2 - \frac{1}{\lambda_{j0}}(x_j -  \mu_{j0})^2 \right) \right].  \label{eq:bayesGauss}
\end{align} Here
$Q_J^{G}(X)$ either converges to a random variable almost surely if  $\sum_{j \ge 1}  (\mu_{j1} - \mu_{j0})^2 / \lambda_{j0} < \infty$ and $\sum_{j \ge 1}(\lambda_{j0} / \lambda_{j1} - 1)^2 < \infty$, or otherwise diverges to $\infty$ or $-\infty$ almost surely, as $J \toinf$. More details about the properties of $Q_J^{G}(X)$ can be found in Lemma~\autoref{lem:gen}  in appendix~\ref{app:thmGen}. It is apparent that \eqref{eq:bayesGauss} is the quadratic discriminant rule using the first $J$ projection scores, which is the Bayes rule for multivariate Gaussian data with different covariance structures.  If further $\lambda_{j0} = \lambda_{j1},\, j=1,2,\ldots$ then one has equal covariances and  \eqref{eq:bayesGauss} reduces to the functional linear discriminant \citep{hall:12:2}.

Because our method does not assume Gaussianity and allows for densities $f_{jk}$ of general form  in \eqref{eq:fr_trunc}, we may expect better performance than Gaussian-based functional classifiers when the population is non-Gaussian. In practice the projection score densities are estimated nonparametrically by kernel density estimation \citep{silv:86} or in the alternative nonparametric regression version by kernel regression \citep{nada:64, wats:64}, as described in \autoref{s:est}. %We discuss the asymptotic properties of the nonparametric estimated classifier in \autoref{s:theory}. 

%Our method reduces to the centroid classifier described in \cite{hall:12:2} when the covariance operators are common to the two population. If the covariance operators are different but with the same eigenfunctions, our method corresponds to the functional quadratic classifier for functional data as was recently proposed by \cite{gale:15}. In Section~\ref{s:theory} we show necessary and sufficient conditions for the truncated classifier $\one\{Q_J(x) \ge 0\}$ to achieve perfect classification as $J \toinf$. 

\section{Estimation} \label{s:est}
We first estimate the common eigenfunctions by pooling data from the both groups to obtain a joint covariance estimate.  Since we assume that the eigenfunctions are the same, while eigenvalues and thus  covariances may  differ, we can write $G_k(s, t) = \Cov(X^{[k]}(s), X^{[k]}(t)) = \sum_{j=1}^\infty \lambda_{jk} \phi_{j}(s) \phi_{j}(t)$ where the $\phi_j$ are the \CE{s}. We define the joint covariance operator  $G = \pi_0 G_0 + \pi_1 G_1$. Then $\phi_j$ is also the $j$th eigenfunction of $G$ with eigenvalue  $\lambda_j = \pi_0\lambda_{j0} + \pi_1\lambda_{j1}$. 

Assume we have $n = n_0 + n_1$ functional predictors $X^{[0]}_1, \dots, X^{[0]}_{n_0}$ and $X^{[1]}_1, \dots, X^{[1]}_{n_1}$ from $\Pi_0$ and $\Pi_1$, respectively. 
In practice, the assumption that functional data for which one wants to construct classifiers are fully observed is often unrealistic. Rather, one has available  dense observations that have been taken on a regular or irregular design, possibly with some missing observations, where the measurements are contaminated with independent measurement errors that have zero mean and finite variance. In this case, we first smooth the discrete observations to obtain a smooth estimate for each trajectory,  using local linear kernel smoothers, and then regard the smoothed trajectory as a fully observed functional predictor. In our theoretical analysis, we justify this approach and show that we obtain the same asymptotic classification results as if we had fully observed the true underlying random functions.  Details about the pre-smoothing and the resulting classification will be given right before Theorem~\autoref{thm:presmooth} in \autoref{s:theory}, where this theorem provides  theoretical justifications for the pre-smoothing approach by establishing asymptotic equivalence to the case of fully observed functions, under suitable regularity conditions. 

We estimate the mean and covariance functions by $\hmu_k(t)$ and $\hG_k(s, t)$, the sample mean and sample covariance function under group $k$, respectively, and estimate $\pi_k$ by $\hpi_k = n_k / n$. Setting $\hG(s,t) = \hpi_0\hG_0(s, t) + \hpi_1\hG_1(s,t)$ and denoting the $j$th eigenvalue-eigenfunction pair of $\hG$ by $(\hlambda_j, \hphi_j)$, we obtain and represent the projections for a generic functional observation by $\hxi_j = \int_\cT X(t) \hphi_j(t) \d t$, $j = 1, \dots, J$. 
We denote the $j$th estimated projection score of the $i$th observation in group $k$ by $\hxi_{ijk}$. The eigenvalues $\lambda_{jk}$ are estimated by $\hlambda_{jk} = \int_\cT \int_\cT \hG_k(s, t) \hphi_j(s) \hphi_j(t) \d s \d t$, which is motivated by $\lambda_{jk} = \int_\cT \int_\cT G_k(s, t) \phi_j(s) \phi_j(t) \d s \d t$,
and the pooled eigenvalues by $\hlambda_j = \hpi_0\hlambda_{j0} + \hpi_1\hlambda_{j1}$.  We estimate the $j$th projection scores $\mu_{jk}$ of $\mu_k(t)$ by $\hmu_{jk} = \int_\cT \hmu_k(t)\hphi_j(t) \d t$. We observe that $\mu_k$, $G_k$, $\phi_j$, and $\lambda_{jk}$ will be consistently estimated, with details in appendix~\ref{app}. 

%Assuming the functional observations are Gaussian, we can use the densities of $N(\hmu_{jk}, \hlambda_{jk})$ as $\hf_{jk}(\cdot)$. Without Gaussian assumption, 
We then proceed to obtain nonparametric estimates of the  densities for each of the projection scores. For this, we use kernel density estimates, applied to the sample projection scores from group $k$. The kernel density estimate \citep{silv:86} for the $j$th component in group $k$ is given by 
\begin{equation} \label{eq:npd}
\hat{f}_{jk}(u) = \frac{1}{n_kh_{jk}} \sum_{i=1}^{n_k} K\Big(\frac{u-\hxi_{ijk}}{h_{jk}}\Big),
\end{equation}
where $u \in \bbR$, and $h_{jk} = h\sqrt{\lambda_{jk}}$ is the bandwidth adapted to the variance of the projection score. The bandwidth multiplier $h$ is the same for all projection density estimates and will be specified in \autoref{s:theory} and \autoref{s:imp}. These estimates then lead to estimated density ratios $\hat{f}_{j1}(u)/\hat{f}_{j0}(u)$.

%This method may be removed due to similarity to the density estimate approach (merely the bandwidths for group 0 and 1 are set to be the same), and poorer simulation performance. If we retain it we may as well just say the bandwidth for smoothing $\Pi_1$ and $\Pi_0$ densities are chosen to be the same for the simplicity of presentation. 
An alternative approach for estimating the density ratios is via nonparametric regression. This is motivated by the Bayes theorem, as follows, 
\be \label{eq:npr}
\frac{f_{j1}(u)}{f_{j0}(u)} &=& \frac{P(Y=1|\xi_j=u) p_j(u) / P(Y=1)}{P(Y=0|\xi_j=u) p_j(u) / P(Y=0)}  \nonumber \\
 &=& \frac{P(Y=1|\xi_j=u) / \pi_1}{P(Y=0|\xi_j=u) / \pi_0} = \frac{\pi_0 P(Y=1|\xi_j=u)}{\pi_1(1 - P(Y=1|\xi_j=u))},
\ee
where $p_j(\cdot)$ is the marginal density of the $j$th projection. This reduces the construction of nonparametric Bayes classifiers to a sequence of nonparametric regressions $E(Y|\xi_j=u)=P(Y=1|\xi_j=u)$. These  again can be implemented by  a kernel method  \citep{nada:64, wats:64}, smoothing the scatter plots of the pooled estimated scores $\hxi_{ijk}$ of group $k = 0, 1$, which leads to the nonparametric estimators 
\[
\hE(Y|\xi_{j} = u) = \frac{\sum_{k=0}^1\sum_{i=1}^{n_k} k K(\frac{u-\hxi_{ijk}}{h_{j}}) }{\sum_{k=0}^1\sum_{i=1}^{n_k} K(\frac{u-\hxi_{ijk}}{h_{j}}) },
\]
where $h_j = h \sqrt{(\lambda_{j0} + \lambda_{j1}) / 2}$ is the bandwidth. This results in estimates  $\hat{E}(Y|\xi_j=u) = \hat{P}(Y=1|\xi_j=u)$ that we plug-in at the right hand side of  \eqref{eq:npr}, which then yields an alternative estimate of the density ratio, replacing the two kernel density estimates $\hat{f}_{j1}(u), \, \hat{f}_{j0}(u)$ by just one nonparametric regression estimate
$\hat{E}(Y|\xi_j=u)$. 

Writing $\hx_j = \int_{\cT} x(t) \hphi_j(t) \d t$, the estimated criterion function based on kernel density estimate is thus 
\begin{equation} \label{eq:Qest}
\hQ_J(x) = \log\frac{\hpi_1}{\hpi_0} + \sum_{j\le J} \log \frac{\hf_{j1}(\hx_j)}{\hf_{j0}(\hx_j)}, 
\end{equation}
while the estimated criterion function based on kernel regression is 
\begin{equation} \label{eq:QestR}
\hQ_J^R(x) = \log\frac{\hpi_1}{\hpi_0} + \sum_{j\le J} \log \frac{\hpi_0 \hE(Y|\xi_{j} = u) }{\hpi_1 [1 - \hE(Y|\xi_{j} = u)]}.
\end{equation}

\section{Theoretical Results} \label{s:theory}
In this section we present the asymptotic equivalence of the estimated version of the Bayes classifiers to the true one. For the first three main results in Theorems \autoref{thm:nonpar}-\autoref{thm:nonparTrue} we assume fully observed functional predictors, and then in Theorem~\autoref{thm:presmooth} we show that these results can be extended to the practically more relevant case where predictor functions are not fully observed, but are only indirectly observed through densely spaced noisy measurements. Following \cite{hall:12:2}, we use the term ``perfect classification'' to mean the misclassification rate approaches zero as more projection scores are used, and proceed to give conditions for the proposed nonparametric  Bayes classifiers to achieve perfect classification. All proofs are in the appendix~\ref{app}.

For theoretical considerations only we assume the following simplifications that can be easily bypassed.   Without loss of generality we denote the mean functions of $\Pi_0$ and $\Pi_1$ as 0 and $\mu(t)$, respectively, since we can subtract the mean function of $\Pi_0$ from all samples, whereupon  $\mu(t)$ becomes the difference in the mean functions, and $\mu_j = \int_{\cT} \mu(t)\phi_j(t) \d t$ stands for the $j$th projection of the mean function. We also assume $\pi_0 = \pi_1$ and $n_0 = n_1$. We use a common multiplier $h$ for all bandwidths $h_{jk} = h\sqrt{\lambda_{jk}}$ in the kernel density estimates and $h_j = h \sqrt{(\lambda_{j0} + \lambda_{j1}) / 2}$ in the kernel regression estimates, for all $j \ge 1$ and $k = 0, 1$.

We need the following assumptions:
\begin{itemize}
\item[(A1)] The covariance operators $G_k(s, t)$ under $\Pi_0$ and $\Pi_1$ have common eigenfunctions;
\item[(A2)] For all $j \ge 1$, the projection scores $\xi_j$ onto the common eigenfunctions $\phi_j$ are independent under $\Pi_0$ and $\Pi_1$, and their densities exist.
\end{itemize}

The common eigenfunction assumption (A1) means the covariance functions $G_k(s, t)$ under $\Pi_0$ and $\Pi_1$ can be decomposed as 
\[
G_k(s, t) = \Cov(X^{[k]}(s), X^{[k]}(t)) = \sum_{j=1}^\infty \lambda_{jk} \phi_{j}(s) \phi_{j}(t),
\]
where the $\phi_j$ are the \CE{s} and $\lambda_{jk}$ are the associated eigenvalues. This means that the major modes of variation are assumed to be the same for both populations, while the variances in the \CE{} directions might change. In practice, the \CE{} assumption allows for meaningful comparisons of the modes of variation between groups, as it makes it possible to reduce such comparisons to comparisons of the functional principal components, as suggested by \cite{coff:11}; for our analysis,  the \CE{s} are convenient projection directions, and are assumed to be such that the projection scores become independent, as is for example the case if predictor processes satisfy the more restrictive  Gaussian assumption.   The \CE{} assumption is weaker than the shared covariance assumption as it allows for  different eigenvalues between groups and thus for different covariance operators across groups.

%With this assumption one can also obtain more stable estimate of the eigenfunctions due to pooled information from both groups than estimating them only within each group. %\cite{benk:09} and \cite{boen:10} discuss the estimation, asymptotics and inference for \CE{s}. 

Theorem~\ref{thm:nonpar} below  states $\hQ_J(x)$ as in \eqref{eq:Qest} is asymptotically equivalent to $Q_J(x)$ as in \eqref{eq:fr_trunc}, for all $J$.  We define the kernel density estimator using the true projection scores $\xi_{ijk}=\int_{\cT} X^{[k]}_i(t)\phi_j(t) \d t$ as
\[
\barf_{jk}(u) = \frac{1}{n_kh_{jk}} \sum_{i=1}^{n_k} K\Big(\frac{u-\xi_{ijk}}{h_{jk}}\Big).
\]
Let $g_{jk}$ be the density functions of the (standardized) FPCs $\xi_j / \sqrt{\lambda_{j0}}$ when $k=0$ and that of $(\xi_j - \mu_j) / \sqrt{\lambda_{j1}}$ when $k=1$, $\hg_{jk}$ be the kernel density estimates of $g_{jk}$ using the estimated FPCs, and $\barg_{jk}$ be the kernel density estimates using the true FPCs, analogous to $\hf_{jk}$ and $\barf_{jk}$. 
\cite{hall:10} provide the uniform convergence rate of $\hg_{jk}$ to $\barg_{jk}$ on a compact domain, with  detailed proof available in \cite{hall:11}, and our derivation utilizes this result.  

We make the following assumptions (B1)--(B5) for $k=0, 1$, in which (B1)--(B4) parallel  assumptions (3.6)--(3.9) in \cite{hall:10}, namely
\begin{itemize}
\item[(B1)] For all large $C > 0$ and some $\delta> 0$, $\sup_{t \in \cT} E_{\Pi_k}\{|X(t)|^C\} < \infty$ and $\sup_{s,t\in \cT: s \ne t} E_{\Pi_k}[\{|s - t|^{-\delta} |X(s) - X(t)|\}^C ] < \infty$; 
\item[(B2)] For each integer $r \ge 1$, $\lambda_{jk}^{-r} E_{\Pi_k}\{\int_{\cT}(X(t) - E_{\Pi_k}X(t)) \phi_j(t) \d t \}^{2r}$ is bounded uniformly in $j$;
\item[(B3)] The eigenvalues $\{\lambda_j\}_{j=1}^\infty$ are all different, and so are the eigenvalues in each of the sequences $\{\lambda_{jk}\}_{j=1}^\infty$, for $k = 0, 1$; 
\item[(B4)] The densities $g_{jk}$ are bounded and have a bounded derivative; the kernel $K$ is a symmetric, compactly supported density function with two bounded derivatives; for some $\delta > 0, h= h(n) = O(n^{-\delta})$ and $n^{1-\delta}h^3$ is bounded away from zero as $n\toinf$. 
\item[(B5)] The densities $g_{jk}$ are bounded away from zero on any compact interval within their respective support, i.e. for all compact intervals $\cI \subset \Supp(g_{jk})$, $\inf_{x_j \in \cI} g_{jk}(x_j) > 0$ for $k = 0, 1$ and $j \ge 1$. 
\end{itemize}

Note that (B1) is a H\"older continuity condition for the process $X(t)$, which is a slightly modified version of a condition in \cite{hall:06:1} and \cite{hall:09}, and that (B2) is satisfied if the standardized FPCs have moments of all orders that are uniformly bounded. In particular, Gaussian processes satisfy (B2) since the standardized FPCs identically follow the standard normal distribution. Recall that the $\lambda_j$ in (B3) are the eigenvalues of the pooled covariance operator, and (B3) is a standard condition \citep{bosq:00}.   (B4) and  (B5) are needed for constructing consistent estimates for the density quotients.
\begin{Theorem} \label{thm:nonpar}
Assuming (A1), (A2), and (B1)--(B5), for any $\epsilon > 0$ there exist a set $S$ with $P(S) > 1 - \epsilon$ and a sequence $J = J(n, \epsilon) \rightarrow \infty$ such that $P(S \cap \{\one\{\hQ_J(X) \ge 0 \} = Y\}) - P(S \cap \{\one\{Q_J(X) \ge 0 \} = Y\}) \tozero$ as $n \toinf$.
\end{Theorem}

Theorem~\autoref{thm:nonpar} provides the asymptotic equivalence of the estimated classifier based on the kernel density estimates  \eqref{eq:npd} and the true Bayes classifier. This  implies that it is sufficient to  investigate the asymptotics of the true Bayes classifier to establish asymptotic perfect classification. The following theorem establishes an analogous result about the equivalence of the estimated classifier based on kernel regression and the true Bayes classifier.

\begin{Theorem} \label{thm:nonparR}
Assuming (A1), (A2), and (B1)--(B5), for any $\epsilon > 0$ there exist a set $S$ with $P(S) > 1 - \epsilon$ and a sequence $J = J(n,\epsilon) \rightarrow \infty$ such that $P(S \cap \{\one\{\hQ_J^R(X) \ge 0 \} = Y\}) - P(S \cap \{\one\{Q_J(X) \ge 0 \} = Y\}) \tozero$ as $n \toinf$.
\end{Theorem}

%\begin{gather*}
%\zeta_j \distas{\Pi_0} N(0, 1), \quad \zeta_j \distas{\Pi_1} N(m_j, r_j\inv), \text{ and } \\
%Q_J^G(X) = \sum\onetoJ[\log r_j - r_j(\zeta_j-m_j)^2 + \zeta_j^2],
%\end{gather*}
%where $\distas{\Pi_k}$ means the distribution under group $k$. 

Our next result shows that the proposed nonparametric Bayes classifiers achieve perfect classification under certain conditions. Intuitively, the following theorem describes when the individual pieces of evidence provided by each of the independent projection scores add up strong enough for perfect classification. Let $m_j = \mu_{j} / \sqrt{\lambda_{j0}}$ and $r_j = \lambda_{j0} / \lambda_{j1}$. We impose the following conditions on the standardized FPCs: 

\begin{itemize}
\item[(C1)] The densities $g_{j0}(\cdot)$ and $g_{j1}(\cdot)$ are uniformly bounded for all $j\ge 1$.
\item[(C2)] The first four moments of $\xi_j / \sqrt{\lambda_{j0}}$ under $\Pi_0$ and those of  $(\xi_j - \mu_j) / \sqrt{\lambda_{j1}}$ under $\Pi_1$ are uniformly bounded for all $j \ge 1$. %by a constant $C_M > 0$.
\end{itemize}

%Bounding the misclassification error of the optimal criterion function $Q_J(x)$ as in (6) by that of the suboptimal criterion function $Q^G_J(x)$ based on Gaussian assumptions as in (8), we have the following Theorem of Theorem~\autoref{lem:gen} which states a sufficient condition for our Bayes classifier $\one\{Q_J(x) \ge 0\}$ to achieve perfect classification:\

\begin{Theorem} \label{thm:nonparTrue}
Assuming (A1), (A2), and (C1)--(C2), the Bayes classifier $\one\{Q_J(x) \ge 0\}$ achieves perfect classification if $\sum_{j\ge 1} (r_j - 1)^2 = \infty$ or $\sum_{j\ge 1} m_j^2 = \infty$,  as $J \toinf$. 
\end{Theorem}

Note that in general the conditions for perfect classification in Theorem~\autoref{thm:nonparTrue} are not necessary. The general case that we study here has the following interesting feature. When $\Pi_1$ and $\Pi_0$ are non-Gaussian, perfect classification may occur even if the mean and covariance functions under the two groups are the same, because one has infinitely many projection scores to obtain information, each possibly having strong discrimination power due to the different shapes of distributions under different groups. 

Consider the following example. Let the projection scores $\xi_j$ be independent random variables with mean 0 and variance $\nu_j$ that follow normal distributions under $\Pi_1$ and Laplace distributions under $\Pi_0$. Then 
\begin{align}
Q_J(X) & = \sum\onetoJ\log \frac{\frac{1}{\sqrt{2\pi\nu_j}} \exp(- \frac{\xi_j^2}{2\nu_j} )}{\frac{1}{\sqrt{2\nu_j}} \exp(-\frac{|\xi_j|}{\sqrt{\nu_j / 2}} )} \nonumber \\
& = \sum\onetoJ \left(-\half\log\pi - \frac{\xi_j^2}{2\nu_j} + \sqrt{2}|\xi_j|/\sqrt{\nu_j}\right). \label{eq:Qexample}
\end{align}
Since centered normal and Laplace distributions form a scale family, we have that $\zeta_j \coloneqq \xi_j / \sqrt{\nu_j}$ have a common standard distribution $\zeta_{0k}$ under $\Pi_k$, irrespective of $j$. Denoting the summand of \eqref{eq:Qexample} by $S_j$, this implies $S_j =  -(\log\pi + \zeta_j^2)/2 + \sqrt{2}|\zeta_j|$ are independent and identically distributed. Note that $E_{\Pi_0}(S_1) = (-\log\pi + 1)/2 + 1 < 0$, $E_{\Pi_1}(S_1) = -(\log\pi + 1)/2 + 2/\sqrt{\pi} > 0$, and $S_1$ has finite variance under either population. So the misclassification error under $\Pi_0$ is 
\begin{align*}
P_{\Pi_0}(Q_J(X) > 0) & = P_{\Pi_0}\left(\sum\onetoJ S_j - E_{\Pi_0}[\sum\onetoJ S_j]> - E_{\Pi_0}[\sum\onetoJ S_j]\right) \nonumber \\
& \le \frac{\Var_{\Pi_0}(\sum\onetoJ S_j)}{[E_{\Pi_0}(\sum\onetoJ S_j)]^2} \\
& = \frac{J\Var_{\Pi_0}(S_1)}{J^2E_{\Pi_0}(S_1)^2} \tozero \quad \text{ as } J\toinf,
\end{align*}
where the inequality is due to Chebyshev's inequality and the last equality is due to $S_j$ are identically and independently distributed. Similarly, the misclassification error under $\Pi_1$ also goes to zero as $J \toinf$. Therefore perfect classification occurs under this non-Gaussian case where both the mean and covariance functions are the same. This provides a case where 
attempts 
at classification under Gaussian assumptions are doomed, as mean and covariance functions are the same between the groups.  

In practice we observe only discrete and noisy measurements 
\begin{equation*}
W_{ikl} = \Xki(t_{ikl}) + \varepsilon_{ikl}
\end{equation*}
for the $i$th functional predictors $\Xki$ in group $k$, for $l = 1, \dots, m_{ik}$, where $m_{ik}$ is the number of measurements per curve. We smooth the noisy measurements by local linear kernel smoothers and obtain $\tXki$, targeting the true predictor $\Xki$. More precisely, for each $t \in \cT$ we let $\tXik(t) = \hbeta_0$, where
\[
(\hbeta_0, \hbeta_1) = \argmin_{\beta_0, \beta_1} \sum_{l=1}^{m_{ik}} K_0(\frac{t - t_{ikl}}{w_{ik}}) [W_{ikl} - \beta_0 - \beta_1(t - t_{ikl})]^2,
\]
$K_0$ is the kernel, and $w_{ik}$ is the bandwidth  for pre-smoothing. 
We let $\bar\tX^{[k]}$ and $\tGk$ be the sample mean and covariance functions of the smoothed predictors in group $k$, and $\tG(s, t) = \pi_0\tG_0(s,t) + \pi_1\tG_1(s,t)$. Also, let $\tphi_{j}(t)$ be the $j$th eigenfunction of $\tG$ for $j=  1, 2, \dots$.
%and $\tlambda_{jk}$ be the corresponding eigenvalue of group $k$, . 
Denote $\txi_{jk} = \int_{\cT} \tXik (t) \tphi_j(t) \d t$ and $\tx_{j} = \int_{\cT} x(t) \tphi_{j}(t) \d t$ as the projection score for a random or fixed function onto $\tphi_j$. Then we use kernel density estimates
\begin{equation} \label{eq:}
\tf_{jk}(u) = \frac{1}{n_k h_{jk}} \sumink K(\frac{u-\txi_{ijk}}{h_{jk}})
\end{equation}
analogous to \eqref{eq:npd}. 

To obtain theoretical results under pre-smoothing, we make regularity assumptions (D1)--(D4), which parallel assumptions (B2)--(B4) in the supplementary material of \cite{kong:16}:

\begin{itemize}
\item[(D1)] For $k=0, 1$, $\Xk$ is twice continuously differentiable on $\cT$ with probability 1, and $\int_\cT E\{(\Xktwo(t)^2 \} \d t < \infty$
\item[(D2)] For $i=1, \dots, n$ and $k=0, 1$, $\{t_{ikl}: l = 1, \dots, m_{ik}\}$ are considered deterministic and ordered increasingly. There exist design densities $u_{ik}(t)$ which are uniformly smooth over $i$ satisfying $\int_{\cT} u_{ik}(t) \d t = 1$ and $0 < c_1 < \inf_i\{\inf_{t\in T} u_{ik}(t)\} < \sup_i\{\sup_{t\in T} u_{ik}(t) \} < c_2 < \infty$ that generate $t_{ikl}$ according to $t_{ikl} = U_{ik}^{-1}\{l/(m_{ik} + 1) \}$, where $U_{ik}\inv$ is the inverse of $G_{ik}(t) = \int_{-\infty}^{t} u_{ik}(s) \d s$. 
\item[(D3)] For each $k=0, 1$, there exist a common sequence of bandwidth $w$ such that $0 < c_1 < \inf_i w_{ik} / w < \sup_i w_{ik} / w < c_2 < \infty$, where $w_{ik}$ is the bandwidth for smoothing $\tXik$. The kernel function $K_0$ for local linear smoothing is twice continuously differentiable and compactly supported. 
\item[(D4)] Let $\delta_{ik} = \sup\{t_{ik,l+1} - t_{ikl}: l=0, \dots, m_{ik} \}$ and $m = m(n) = \inf_{i=1, \dots, n;\, k=0, 1} m_{ik}$. $\sup_{i, k} \delta_{ik} = O(m\inv)$, $w$ is of order $m^{-1/5}$, and $m h^5 \rightarrow \infty$, where $h$ is the common bandwidth multiplier in kernel density estimates.
\end{itemize}

Let $\tQ_J$ be the classifier using $J$ components analogous to $\hQ_J$ in \eqref{eq:Qest}, but with kernel density estimates $\tf_{jk}$ constructed with the pre-smoothed predictors. Under the above assumptions, we obtain an extension of Theorem~\autoref{thm:nonpar}.

\begin{Theorem} \label{thm:presmooth}
Assuming (A1), (A2), (B1)--(B5), and (D1)--(D4), for any $\epsilon > 0$ there exist a set $S$ with $P(S) > 1 - \epsilon$ and a sequence $J = J(n, \epsilon) \rightarrow \infty$ such that $P(S \cap \{\one\{\tQ_J(X) \ge 0 \} = Y\}) - P(S \cap \{\one\{Q_J(X) \ge 0 \} = Y\}) \tozero$ as $n \toinf$.
\end{Theorem}

\section{Numerical Properties}
\subsection{Practical Considerations} \label{s:imp}
We propose three practical implementations for estimating the projection score densities $f_{jk}(\cdot)$ that will be compared in our data illustrations, along with other previously proposed functional classification methods. All of these involve the choice of tuning parameters (namely bandwidths and number of included components) and we describe in the following how these are specified.

Our first implementation is the nonparametric density classifier as in \eqref{eq:Qest}, where one estimates the density of each projection by applying kernel density estimators to the observed sample scores as in \eqref{eq:npd}. For these kernel estimates we use a Gaussian kernel and the bandwidth multiplier $h$ is chosen by 10-fold cross-validation (CV), minimizing the misclassification rate.

The second implementation is nonparametric regression as in \eqref{eq:QestR}, where we apply kernel smoothing \citep{nada:64, wats:64} to the scatter plots of the pooled estimated scores and group labels. For each scatter plot, a Gaussian kernel is used and the the bandwidth multiplier $h$ is also chosen by 10-fold CV.

Our third implementation, referred to as Gaussian method and included mainly for comparison purposes, assumes each of the projections to be normally distributed with mean and variance estimated by the sample mean $\hmu_{jk} = \sumink\hxi_{ijk} / n_k$ and sample variance $\hlambda_{jk} = \sumink (\hxi_{ijk} - \hmu_{jk})^2 / (n_k-1)$ of $\hxi_{ijk}$, $i=1, \dots, n$. We then  use the density of $N(\hmu_{jk}, \hlambda_{jk})$ as $\hf_{jk}(\cdot)$. This Gaussian implementation differs from the quadratic discriminant implementation discussed for example in \ci{hall:13:1}, 
 as in our approach we always force the projection directions for the two populations to be the same. This has the practical advantage of providing more stable estimates for the eigenfunctions and is a prerequisite for constructing nonparametric Bayes classifiers for functional predictors. 
 
For numerical stability, if the densities are zero we insert a  very small lower threshold (100 times the gap between 0 and the next double-precision floating-point number). Finally, the number of projections $J$ used in our implementations is chosen by 10-fold CV (together with the selection of $h$ for the nonparametric classifiers).

\subsection{Simulation Results} \label{s:sim}
We illustrate our Bayes classifiers in several simulation settings. In each setting we generate $n$ training samples, each having $1/2$ chance to be from $\Pi_0$ or $\Pi_1$. The samples are generated by $X_i^{[k]}(t) = \mu_k(t) + \sum_{j=1}^{50} A_{ijk} \phi_j(t)$, $i = 1, \dots, n_k$, where $n_k$ is the number of samples in $\Pi_k$, $k=0,\, 1$. Here the $A_{ijk}$ are independent random variables with mean 0 and variance $\lambda_{jk}$, which are generated under two distribution scenarios: Scenario A, the $A_{ijk}$ are normally distributed, and Scenario B, the $A_{ijk}$ are centered exponentially distributed. In both scenarios, the $\phi_j$ are the $j$th function in the Fourier basis, where $\phi_1(t)=1$, $\phi_2(t) = \sqrt{2}\cos(2 \pi t)$, $\phi_3(t) = \sqrt{2}\sin(2 \pi t)$, etc., $t \in [0,\, 1]$. We set $\mu_0(t) = 0$, and $\mu_1(t) = 0$ or $t$ for the same or the different mean scenarios, respectively. The variances of the scores under $\Pi_0$ are $\lambda_{j0} = e^{-j/3 }$, and those under $\Pi_1$ are $\lambda_{j1} =e^{-j/3 }$ or $e^{-j/2 }$ for the same or the different variance scenarios, respectively, for $j = 1, \dots, \, 50$. The random functions are sampled at 51 equally spaced time points from 0 to 1, with additional small measurement errors in the form of independent Gaussian noise with mean 0 and variance 0.01 to each observation for both scenarios. We use modest sample sizes of $n = 50$ and $n=100$ for training the classifiers, and 500 samples for evaluating the predictive performance. We repeat each simulation experiment  $500$ times.

We compare the predictive performance of the following functional classification methods: (1) the centroid method \citep{hall:12:2}; (2) the proposed nonparametric Bayes classifier in three versions: Basing estimation on Gaussian densities (Gaussian), nonparametric densities (NPD), or nonparametric regression (NPR), which are the three implementations discussed above; (3) logistic regression; and (4) the functional quadratic discriminant as in \cite{gale:15}. The functional quadratic discriminant was never the winner for any scenario in our simulation study so we omitted it from the tables.

The results for Scenario A are shown in \autoref{tab:norm}, whose upper half corresponds to using the noisy predictors as is, and the lower half corresponds to pre-smoothing the predictors by local linear smoother with CV bandwidth choice. For these cases, the proposed nonparametric Bayes classifiers  are seen to have  superior performance for those scenarios where covariance differences in the populations are present, while the centroid and the logistic methods work best for those cases where  the differences are exclusively in the mean. 

\begin{table}[ht]
\caption{Misclassification rates (in percent) for Scenario A (Gaussian case), with standard deviation for the mean estimate in brackets. The Gaussian, NPD, and NPR methods correspond to the Gaussian, nonparametric density, and nonparametric regression implementations of the proposed Bayes classifiers, respectively. The upper half corresponds to using the functional predictors with noisy measurements as is, and the lower half corresponds to using pre-smoothed  predictors.} 
\label{tab:norm}
\centering
{ \tight
\begin{tabular}{llllllll}
	\hline
	$n$   &  $\mu$   &  $\lambda$ & Centroid                         & Gaussian                         & NPD         & NPR         & Logistic                         \\ \hline
\multicolumn{8}{l}{without pre-smoothing:} \\
	50  & same & diff   & 49.3 (0.12)                      & \textbf{\color{blue}23.8 (0.18)} & 24.5 (0.21) & 26.7 (0.22) & 49.4 (0.12)                      \\
	    & diff & same   & \textbf{\color{blue}40.2 (0.16)} & 41.5 (0.16)                      & 43.4 (0.17) & 42.4 (0.18) & 40.7 (0.16)                      \\
	    & diff & diff   & 37.9 (0.17)                      & \textbf{\color{blue}20.8 (0.18)} & 21.2 (0.2)  & 23.3 (0.22) & 38.8 (0.17)                      \\
	100 & same & diff   & 49.1 (0.13)                      & \textbf{\color{blue}17.2 (0.11)} & 18.6 (0.12) & 20 (0.13)   & 49.3 (0.13)                      \\
	    & diff & same   & \textbf{\color{blue}37.8 (0.13)} & 39.2 (0.13)                      & 41.4 (0.15) & 40.2 (0.16) & 38.3 (0.13)                      \\
	    & diff & diff   & 35.3 (0.14)                      & \textbf{\color{blue}14.6 (0.1)}  & 15.8 (0.1)  & 17.1 (0.12) & 35.8 (0.15)                      \\
	    \hline
	    \multicolumn{8}{l}{with pre-smoothing:} \\
	50  & same & diff   & 48.9 (0.14)                      & \textbf{\color{blue}22.7 (0.17)} & 23.1 (0.2)  & 25.7 (0.21) & 48.9 (0.13)                      \\
	    & diff & same   & 36.5 (0.24)                      & 38.3 (0.22)                      & 40.7 (0.22) & 39.3 (0.23) & \textbf{\color{blue}32.2 (0.26)} \\
	    & diff & diff   & 33.4 (0.25)                      & \textbf{\color{blue}18 (0.16)}   & 18.4 (0.18) & 20.3 (0.2)  & 28.1 (0.26)                      \\
	100 & same & diff   & 48.9 (0.14)                      & \textbf{\color{blue}17.1 (0.11)} & 18.1 (0.12) & 19.4 (0.13) & 49.1 (0.14)                      \\
	    & diff & same   & 29.8 (0.23)                      & 31.6 (0.23)                      & 33.6 (0.25) & 31.9 (0.25) & \textbf{\color{blue}25.4 (0.15)} \\
	    & diff & diff   & 27 (0.24)                        & \textbf{\color{blue}13 (0.11)}   & 14 (0.12)   & 14.8 (0.13) & 21.1 (0.14)                      \\ \hline
\end{tabular}
}
\end{table}

%\parbox[t]{2em}{\multirow{4}{1em}{\rotatebox[origin=c]{90}{pre-smooth}}}

In the cases where the covariances differ, the proposed Bayes classifiers have substantial performance advantages over other methods. This is because they take into account both mean and covariance differences between the populations. When the covariances are the same but the means differ, the centroid method is the overall best if we use the noisy predictors while the Gaussian implementation of the proposed Bayes classifiers has comparable performance. This is expected because our method estimates more parameters than the centroid method while both assume the correct model for the simulated data. The quadratic method (not shown) is not performing well for these simulation data because it fails to take into account the common eigenfunction structure. The logistic method gains considerable performance from pre-smoothing, and becomes the winner when only a mean difference is present.

%The nonparametric regression method performs best when the covariance differences are large, and in  almost all cases it outperforms the nonparametric density estimation implementation. 
%Compared to the centroid method, our methods is able to make use of more projections in most cases, which contain discriminating power. 

The simulation results for Scenario B are reported in \autoref{tab:exp}.
The performance of the proposed Bayes classifiers deteriorates somewhat in this scenario but they still perform substantially better than all other methods when covariance differences occur. When there are differences between the covariances, the Gaussian implementation performs the best when the sample size is small, while the nonparametric density implementation performs the best when the sample size is large. 

\begin{table}[ht]
\caption{Misclassification rates (in percent) for Scenario B (exponential case), with standard deviation for the mean estimate in brackets. The upper half corresponds to using the functional predictors with noisy measurements as is, and the lower half corresponds to using pre-smoothed  predictors.}
\label{tab:exp}
\centering
{ \tight
\begin{tabular}{llllllll}
	\hline
	$n$   &  $\mu$   &  $\lambda$ & Centroid                         & Gaussian                         & NPD                              & NPR                              & Logistic                         \\ \hline
	\multicolumn{8}{l}{without pre-smoothing:} \\
	50  & same & diff   & 49 (0.13)                        & \textbf{\color{blue}30.2 (0.19)} & 31.2 (0.22)                      & 33.5 (0.23)                      & 49.2 (0.13)                      \\
	    & diff & same   & \textbf{\color{blue}38.3 (0.21)} & 40.6 (0.21)                      & 39.5 (0.22)                      & 38.6 (0.21)                      & 38.7 (0.23)                      \\
	    & diff & diff   & 35 (0.2)                         & \textbf{\color{blue}23.3 (0.18)} & 23.5 (0.21)                      & 24.3 (0.22)                      & 35.7 (0.22)                      \\
	100 & same & diff   & 48.8 (0.14)                      & 26 (0.13)                        & \textbf{\color{blue}25.4 (0.14)} & 26.7 (0.16)                      & 48.9 (0.13)                      \\
	    & diff & same   & 35.8 (0.16)                      & 38.6 (0.19)                      & 36.3 (0.18)                      & \textbf{\color{blue}35.7 (0.16)} & 35.9 (0.16)                      \\
	    & diff & diff   & 32.4 (0.14)                      & 18.7 (0.13)                      & \textbf{\color{blue}16.7 (0.13)} & 17 (0.14)                        & 32.7 (0.15)                      \\\hline
	    	    \multicolumn{8}{l}{with pre-smoothing:} \\
	50  & same & diff   & 48.5 (0.15)                      & \textbf{\color{blue}28.3 (0.18)} & 29.1 (0.21)                      & 31.4 (0.24)                      & 48.6 (0.14)                      \\
	    & diff & same   & 35 (0.24)                        & 38.4 (0.22)                      & 38 (0.22)                        & 36.5 (0.23)                      & \textbf{\color{blue}30.9 (0.23)} \\
	    & diff & diff   & 30.3 (0.24)                      & \textbf{\color{blue}20.2 (0.18)} & 20.9 (0.22)                      & 21.4 (0.22)                      & 27 (0.23)                        \\
	100 & same & diff   & 48.5 (0.15)                      & 25.1 (0.13)                      & \textbf{\color{blue}24 (0.14)}   & 25 (0.14)                        & 48.4 (0.15)                      \\
	    & diff & same   & 29.2 (0.23)                      & 33.3 (0.23)                      & 32.3 (0.2)                       & 31.1 (0.21)                      & \textbf{\color{blue}25.4 (0.17)} \\
	    & diff & diff   & 26.1 (0.22)                      & 16.5 (0.14)                      & \textbf{\color{blue}14.6 (0.13)} & 14.7 (0.13)                      & 21.6 (0.16)                      \\ \hline
\end{tabular}
}
\end{table}

\subsection{ Data Illustrations} \label{s:data}
We present three data examples to illustrate the performance of the proposed Bayes classifiers for functional data. 
%For each dataset we report 500-repeat 10-fold CV misclassification rate. 
We pre-smooth the yeast data by local linear smoother with CV bandwidth choice since the original observations are quite noisy (shown in \autoref{fig:yeast_adhd_obs}), while for the wine and the ADHD datasets we just use the original curves which are preprocessed and smooth. Following the procedure described in \cite{benk:09}, we  test the common eigenspaces assumption for the first $J=5$ and 20 eigenfunctions and report in \autoref{tab:pv} the p-values obtained from 2000 bootstrap samples. Only one test rejects the null hypothesis that the first $J$ eigenspaces are shared by the two populations at 0.1 significance level, which shows the common eigenfunction assumption is reasonable.

\begin{table}[h!]
\caption{The p-values for testing the common eigenspace assumptions, using $J$ = 5 or 20 eigenfunctions. We report the results for both the original (yeast) and pre-smoothed version (yeast\_pre) version of the yeast gene expression dataset.} 
\label{tab:pv}
\centering
{\small \tight
\begin{tabular}{c|cccccc}
	\hline
	       & ADHD & yeast & yeast\_pre &      wine\_full      & wine\_d1 &  \\ \hline
	J = 5  & 0.31 & 0.57  &    0.15    & \textbf{\blu{0.098}} &   0.31   &  \\
	J = 20 & 0.80 & 0.63  &    0.76    &         0.72         &   0.55   &  \\ \hline
\end{tabular}}
\end{table}
We used repeated 10-fold CV misclassification error rates to evaluate the performance of the classifiers. In order to obtain the correct CV misclassification error rate the selection of the number of components and bandwidth is carried out on only the training data in each CV partition. We repeat the process 500 times and report the mean misclassification rates, and the standard deviations of the mean estimates. The misclassification results for different datasets are shown in \autoref{tab:real}.

As can be seen from \autoref{tab:real}, which contains all results for misclassification rates across the compared methods and data sets, the proposed nonparametric Bayes classifiers and the functional quadratic discriminant perform overall well, indicating that  covariance operator differences contain crucial information for classification. Among the various implementations of the proposed Bayes classifiers, the Gaussian version performs best for these data. Pre-smoothing the predictors slightly improve the misclassification rate for the yeast dataset. We now provide more details about the various data.

\begin{table}[ht]
\caption{CV misclassification rates (in percent) for the three example data. ADHD refers to the attention deficit hyperactivity disorder data. The yeast data refers to cell cycle gene expression time course data, and yeast\_pre refers to the pre-smoothed version of the yeast data. The wine datasets concern the classification of the original spectra (wine\_full) and the first derivative (wine\_d1), which is constructed  by forming difference quotients. }
\label{tab:real}
\centering
\small \tight
\begin{tabular}{lllllll}
	\hline
	Data       & Centroid    & Gaussian                                     & NPD         & NPR                             & Logistic    & Quadratic   \\ \hline
	ADHD       & 41.7 (0.2)  & \textbf{\color{blue}34.1 (0.1)}              & 36.7 (0.2)  & 36.8 (0.2)                      & 47 (0.2)    & 34.6 (0.2)  \\
	yeast      & 20.0 (0.08) & \textbf{\color{blue}12.5 (0.09)}             & 15 (0.1)    & 14.4 (0.1)                      & 20.8 (0.1)  & 14.5 (0.09) \\
	yeast\_pre & 20.7 (0.1)  & \textbf{\color{blue}\color{blue}12.3 (0.06)} & 14.3 (0.09) & 14.1 (0.1)                      & 17.2 (0.09) & 14.4 (0.07) \\
	wine\_full & 6.84 (0.07)  & 5.08 (0.06)                                   & 5.09 (0.06)  & \textbf{\color{blue}4.67 (0.06)} & 7.56 (0.08)  & 5.93 (0.08)  \\
	wine\_d1   & 7.15 (0.07)  & 6.99 (0.06)                                   & 5.75 (0.06)  & \textbf{\color{blue}5.37 (0.06)} & 6.64 (0.07)  & 5.69 (0.07)  \\ \hline
\end{tabular}

\end{table}

The first data example concerns classifying attention deficit hyperactivity disorder (ADHD) from brain imaging data. The data were obtained in 
 the ADHD-200 Sample Initiative Project. ADHD is the most commonly diagnosed behavioral
disorder in childhood, and can continue through adolescence and adulthood. The
symptoms include lack of attention, hyperactivity, and impulsive behavior. We base our analysis on 
filtered preprocessed resting state data from the New York University (NYU) Child
Study Center, called the  anatomical automatic labelling atlas \citep{Mazoyer2002}, which contains 116 Regions of Interests (ROI) that have been fractionated into functional resolution of the original image using nearest-neighbor interpolation to create a discrete labelling value for each pixel of the image. The mean blood-oxygen-level dependent signal in each ROI is depicted for 172 equally spaced time points. We use only subjects for which the ADHD index is in the  lower quartile (defining $\Pi_0$)  or upper quartile (defining $\Pi_1$), with $n_0 = 36$ and $n_1=34$, respectively, and regard the group membership as the binary response to be predicted. The functional predictor is taken to be the average of the mean blood-oxygen-level dependent signals of the 91st to 108th regions, shown in \autoref{fig:yeast_adhd_obs}, corresponding to the cerebellum that has been found to have significant impact on the ADHD index in previous studies \citep{Berquin1998}. % The response to be predicted is the quartile of the ADHD score (lower or upper).

\begin{figure}
\centering
\includegraphics[width=0.45\linewidth]{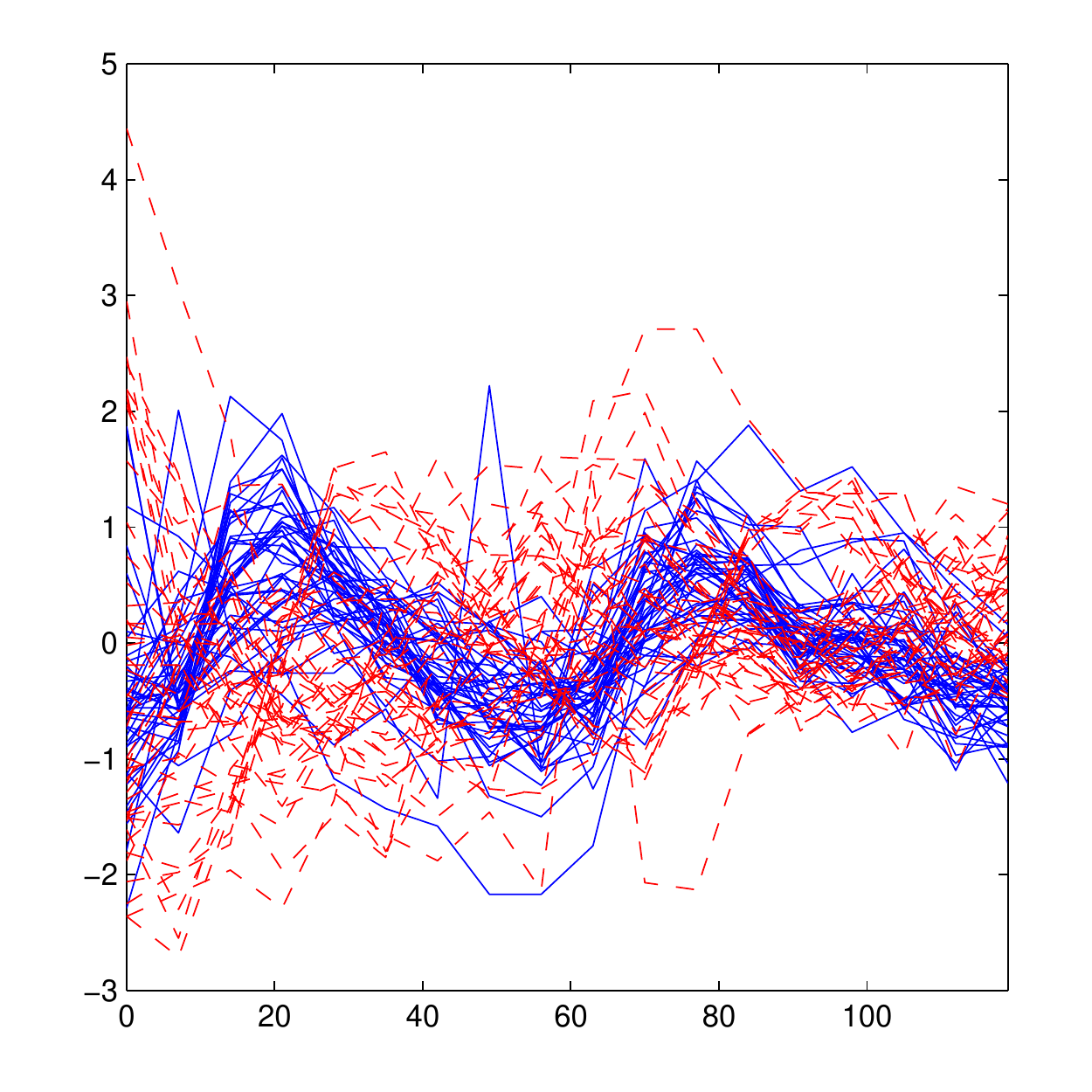}
\includegraphics[width=0.45\linewidth]{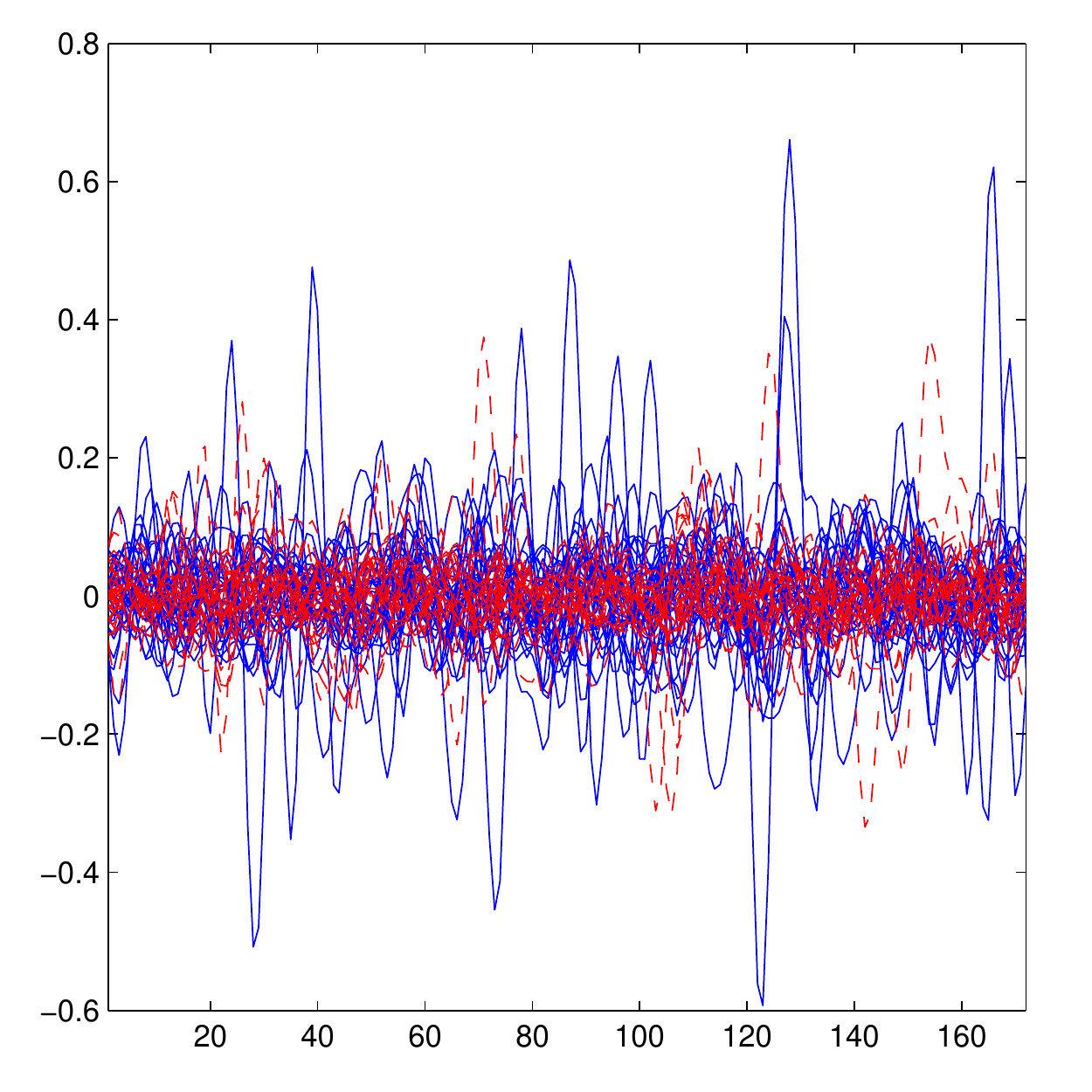}
\caption{The original functional predictors for the yeast (left panel) and Attention Deficit Hyperactivity Disorder (ADHD, right panel) data.  $\Pi_0$ is shown in dashed lines and $\Pi_1$  in solid lines.}
\label{fig:yeast_adhd_obs}
\end{figure}

Our second data example focuses on  yeast gene expression time courses during the cell cycle as predictors, which are described in \cite{spell:98}. The predictors are  gene expression level time courses for   $n=89$ genes, observed at 18 equally spaced time points from 0 minute to 119 minutes. The expression trajectories for genes related to G1 phase regulation of the cell cycle were regarded as $\Pi_1$ ($n_1 = 44$) and all others are regarded as $\Pi_0$ ($n_0 = 45$). The Gaussian  implementation of the proposed Bayes classifiers outperforms the other methods by a margin of at least 2\%, while the functional quadratic discriminant is also  competitive for this classification problem. Pre-smoothing improves the performance of all classifiers except the centroid method. 

In the third example we analyze wine spectra data. These data have been made available by Professor Marc Meurens, Université Catholique de Louvain, at \url{http://mlg.info.ucl.ac.be/index.php?page=DataBases}. The dataset contains a training set of 93 samples and a testing set of 30 samples. We combine the training set and test set into a dataset of $n = 123$. For each sample the mean infrared spectrum on 256 points and the alcohol content are observed. $\Pi_1$ consists of the samples  with alcohol contents  greater than 12 ($n_1 = 78$) and $\Pi_0$ ($n_0 = 45$) of the rest. We consider both the original observations (wine\_full) and the first derivative (wine\_d1), which is constructed  by forming difference quotients. As for the other examples,  the misclassification errors for the various methods are listed in \autoref{tab:real}.

The original functional predictors for the wine example and the mean functions for each group are displayed in the left and the right panel of  \autoref{fig:wine}, respectively. There are clear mean differences between  the two groups, especially around the peaks, for example at $t=180$. We show in \autoref{fig:wine_densities} the kernel density estimates of the first four projection scores, with $\Pi_0$ in dashed lines and $\Pi_1$  in solid lines. Clearly the densities are not normal, and some of them (the first and second projections) appear to be bimodal.  The differences between each pair of densities are not limited to location and scale, but also manifest themselves in  the shapes of the densities; in the second and fourth plots the density estimate from one group is close to bimodal and the other density is not.
The nonparametric implementations of the proposed Bayes estimators based on nonparametric regression or nonparametric density estimation are capable of reflecting such shape differences and therefore outperform the classifiers based on Gaussian assumptions.

\begin{figure}
\centering
\includegraphics[width=0.45\linewidth]{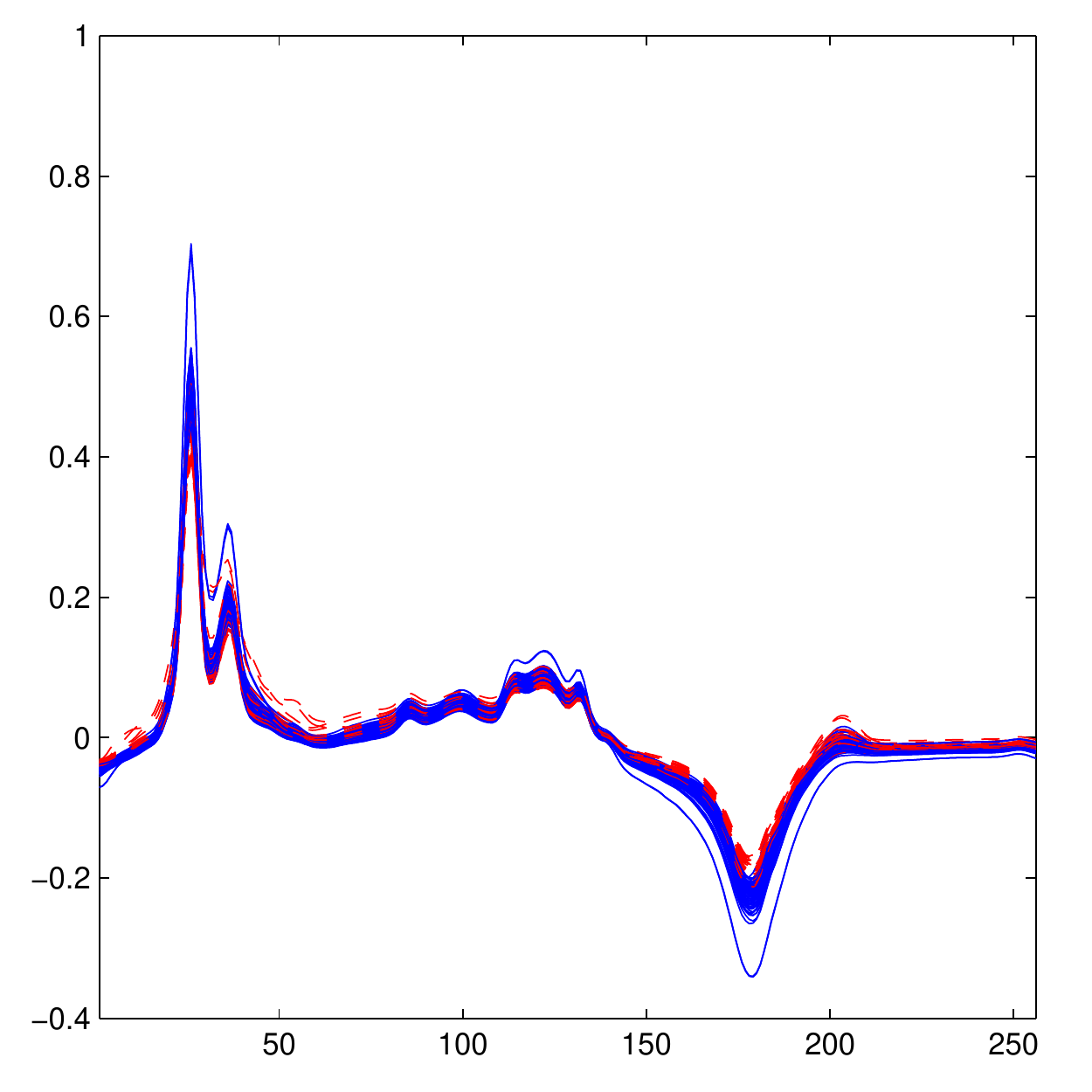}
\includegraphics[width=0.45\linewidth]{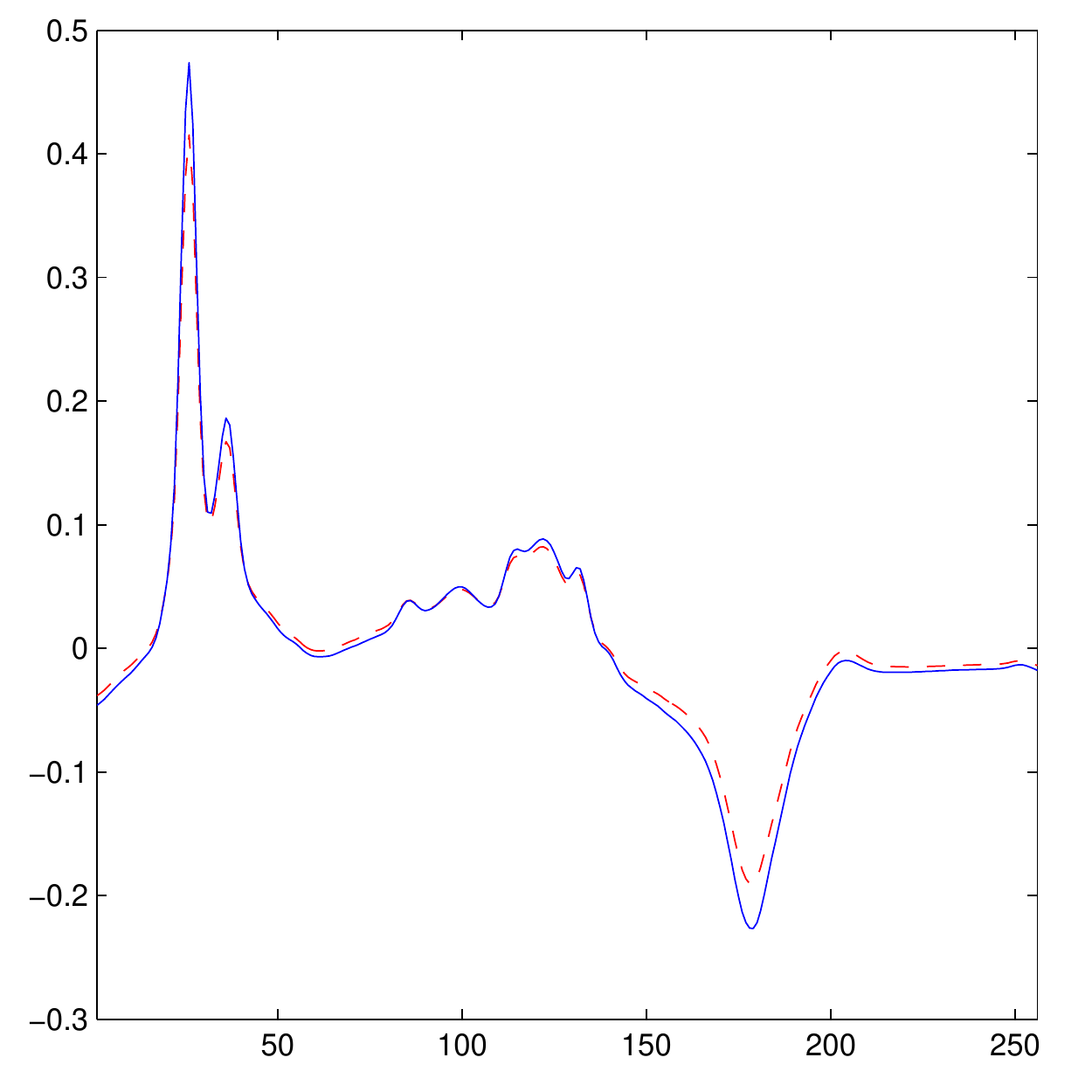}
\caption{The Wine Spectra. The left panel shows the original trajectories and the right panel shows the mean curves for each group. Trajectories of $\Pi_0$ are displayed  in dashed lines and those of $\Pi_1$ in solid lines.}
\label{fig:wine}
\end{figure}

%\begin{figure}
%\centering
%\caption{The Wine Spectra Mean Functions. $\Pi_0$ is shown in dashed lines and $\Pi_1$  in solid lines.} 
%\label{fig:wine_mean}
%\end{figure}

\begin{figure}
\centering
\includegraphics[width=\linewidth]{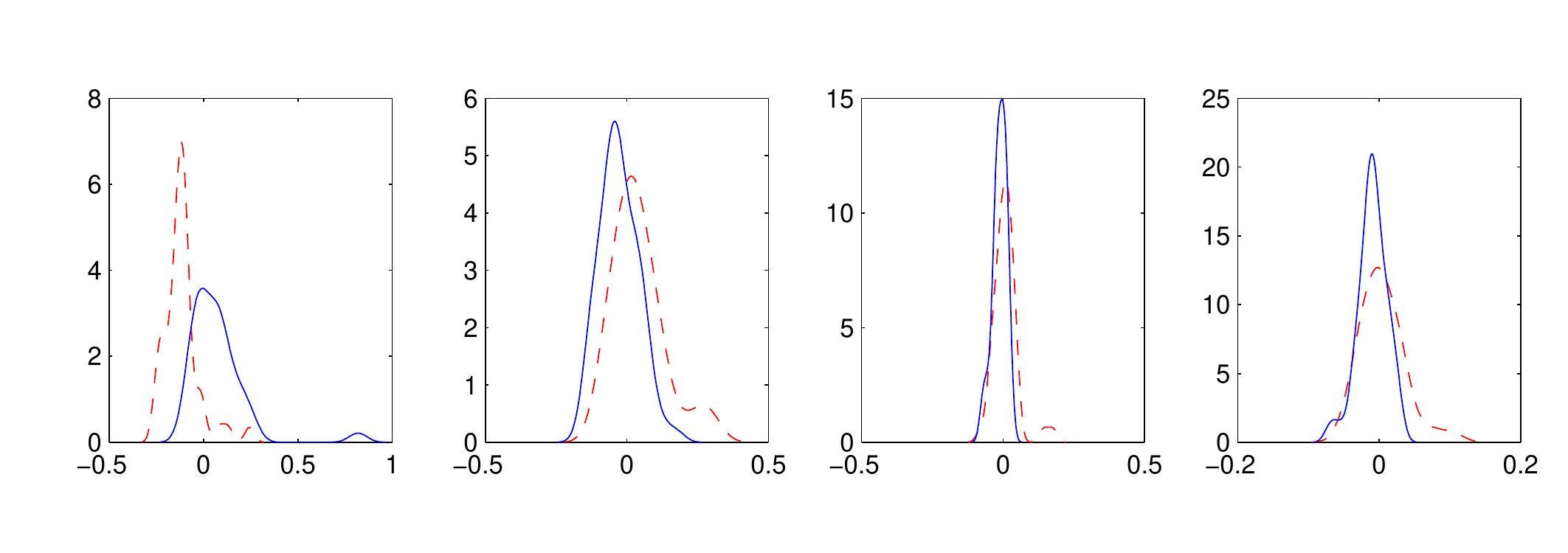}
\caption{Kernel density estimates for the first four projection scores for the wine spectra.  $\Pi_0$ is shown in dashed lines and $\Pi_1$  in solid lines.}
\label{fig:wine_densities}
\end{figure}

In all examples, the quadratic discriminant performs better than the centroid method, suggesting that in these examples there is information contained in the differences between the covariance operators of the two groups to be classified. In the presence of such more subtle differences
and additional shape differences in the distributions of projection scores the proposed nonparametric Bayes methods work particularly well for functional classification. 

% Our method outperforms the quadratic discriminant probably because as we assume \CE{s} the estimates are more stable than those for the quadratic method.
\vspace{0.15in}

%\subsection*{Acknowledgments}

%We are grateful to the referees for helpful comments that led to
%an improved version of the paper.

\appendix

\section{Technical Arguments} \label{app}
For simplicity of presentation we adopt throughout all proofs the simplifying assumptions mentioned in \autoref{s:theory}. We remark that $\hmu_k$, $\hG_k$, $\hphi_{j}$, and $\hlambda_{jk}$ constructed from the sample mean, covariance, eigenfunctions and eigenvalues of the completely observed functions are consistent estimates for their corresponding targets, as per \cite{hall:06:1}. 
 % % % % % % % % % % % % % % % %

\subsection{Proof of Theorem~\ref{thm:nonpar}} \label{app:nonparEst}
Let $\cS(c) = \{x:||x|| \le c \}$ be a bounded set of all square integrable functions for $c > 0$, where $||\cdot||$ is the $L^2$ norm. We will use the following lemma:
\begin{Lemma} \label{lem:hff}
Assuming (B1)--(B4), for any $j \ge 1$, $k = 0, 1$, 
\begin{equation} \label{eq:hff}
\sup_{x \in \cS(c)} |\hf_{jk}(\hx_j) - f_{jk}(x_j)| = O_P(h + (\frac{nh}{\log n})^{-\half}).
\end{equation}\end{Lemma}
\begin{proof}
We prove the statement  for $k=0$; the proof for  $k=1$ is analogous. Observe 
\begin{align}
\sup_{x \in \cS(c)}|\hg_{j0}(\frac{\hx_j}{\sqrt{\hlambda_{j0}}} ) - g_{j0}(\frac{x_j}{\sqrt{\lambda_{j0}}} )| 
& \le 
\sup_{x \in \cS(c)} |\hg_{j0}(\frac{\hx_j}{\sqrt{\hlambda_{j0}}} ) - \barg_{j0}(\frac{x_j}{\sqrt{\lambda_{j0}}} )| + 
\sup_{x \in \cS(c)}|\barg_{j0}(\frac{x_j}{\sqrt{\lambda_{j0}}} ) - g_{j0}(\frac{x_j}{\sqrt{\lambda_{j0}}} )| \nonumber\\
& = o_p((nh)^{-1/2}) + O_p(h + (\frac{nh}{\log n}) ^{-\half}) = O_p(h + (\frac{nh}{\log n}) ^{-\half}), \label{eq:gRate}
\end{align}
where the first rate is due to \cite{hall:10}, and the second to,  for example, \cite{ston:83} or \cite{lieb:96}. Then 
\begin{align}
\sup_{x \in \cS(c)}|\hf_{j0}(\hx_j) - f_{j0}(x_j)| & = \sup_{x \in \cS(c)} |\frac{1}{\sqrt{\hlambda_{j0}}} \hg_{j0}(\frac{\hx_j}{\sqrt{\hlambda_{j0}}}) - \frac{1}{\sqrt{\lambda_{j0}}}g_{j0}(\frac{x_j}{\sqrt{\lambda_{j0}}})| \nonumber \\
& \le 
\sup_{x \in \cS(c)} \left\{\frac{1}{\sqrt{\hlambda_{j0}}} |\hg_{j0}(\frac{\hx_j}{\sqrt{\hlambda_{j0}}} ) - g_{j0}(\frac{x_j}{\sqrt{\lambda_{j0}}} )| + g_{j0}(\frac{x_j}{\sqrt{\lambda_{j0}}}) |\frac{1}{\sqrt{\hlambda_{j0}}} - \frac{1}{\sqrt{\lambda_{j0}}}| \right\} \nonumber \\
& = 
O_p(\sup_{x \in \cS(c)}|\hg_{j0}(\frac{\hx_j}{\sqrt{\hlambda_{j0}}} ) - g_{j0}(\frac{x_j}{\sqrt{\lambda_{j0}}} )|) + 
O_p( |\frac{1}{\sqrt{\hlambda_{j0}}} - \frac{1}{\sqrt{\lambda_{j0}}}| ) \nonumber \\
& = 
O_p(h + (\frac{nh}{\log n}) ^{-\half}), \label{eq:fRate}
\end{align}
where the second equality follows from the consistency of $\hlambda_{j0}$ and boundedness of $g_{j0}$ (B4), and the third equality follows from \eqref{eq:gRate} and the fact that  $\hlambda_{j0}$ converges at a root-$n$ rate. 
\end{proof}

\begin{proof}[Proof of Theorem~\autoref{thm:nonpar}]
For simplicity we consider the case where the supports of $g_{j0}$ and $g_{j1}$ are in common. The case where the supports differ can be proven in two step:
First consider to classify elements $x$ whose projections $x_j$ are in the intersection of supports of $g_{j0}$ and $g_{j1}$. Next consider to classify an element $x$ for which a projection score $x_j$ is not contained in the intersection of the supports, in which case $Q_J(x)$ will be $\pm \infty$ whence $\hQ_J(x)$ will also diverge to $\pm \infty$, respectively, and thus consistency is obtained.

Now fix $\epsilon > 0$. Set $c$ be such that $P(||X|| > c) = P(X \notin \cS(c)) \le \epsilon/2$. First we prove there exists an event $S$ such that $\hQ_J(X) - Q_J(X) \tozero$ on $S$ with $P(S) > 1 - \epsilon$. For $j \ge 1$ and $k= 0, 1$, by Lemma~\autoref{lem:hff} there exists $M_{jk} > 0$ such that the events
\[
S_{jk} \coloneqq \{\sup_{x \in \cS(c)}|\hf_{jk}(\hx_j) - f_{jk}(x_j)| \le M_{jk}(h + (\frac{nh}{\log n}) ^{-\half}) \}
\]
have $P(S_{jk}) \ge 1 - 2 ^{-(j+2)}\epsilon$. Letting $S \coloneqq \left(\bigcap_{j\ge 1, k=0, 1}S_{jk}\right) \cap \left(\bigcap_{j\ge 1, k=0, 1}\{\xi_j\in \text{Supp}(f_{jk}) \}\right) \cap \{||X|| \le c \}$, we have $P(S) \ge 1 - \epsilon$, where Supp means the support of a density. Let $a_n$ be some increasing sequence such that $a_n \toinf$ and $a_n[h + (nh / \log n)^{-\half}] = o(1)$. Define  $\cU_{jk} = \{x: x_j \in \text{Supp}(f_{jk}) \}$, $\cU = \bigcap_{j\ge 1, k=0, 1}\cU_{jk}$, 
\begin{gather} 
d_{jk} = \min(1, \inf_{x \in \cS(c) \cap \cU} f_{jk}(x_j)), \text{ and} \\
J = \sup \left\{ J' \ge 1: \sum_{j\le J',\, k=0, 1} \frac{M_{jk}}{d_{jk}} \le  a_n \right\}. \label{eq:Jnonpar}
\end{gather}
Note that the $d_{jk}$ are finite by (B5), and $J$ is nondecreasing and tends to infinity as $n \toinf$. 
On $S$ we have 
\begin{equation} \label{eq:sumdjk}
\sum\onetoJ\frac{1}{d_{jk}}\sup_{x \in \cS(c)}|\hf_{jk}(\hx_j) - f_{jk}(x_j)| \le \sum\onetoJ \frac{M_{jk}}{d_{jk}} [h + (\frac{nh}{\log n})^{-\half}] \le  a_n [h + (\frac{nh}{\log n})^{-\half}] = o(1),
\end{equation}
where the first and second inequalities are due to the property of $S$ and $J$, respectively, and the last equality is by the definition of $a_n$. 

From \eqref{eq:sumdjk} we infer that on  $S$,
\begin{equation} \label{eq:djk}
\sup_{x \in \cS(c)}|\hf_{jk}(\hx_j) - f_{jk}(x_j)| \le d_{jk}/2 \quad \text{eventually and uniformly for all $j \le J$}.
\end{equation} 
Then it holds on $S$
\begin{align}
|\hQ_J(X) - Q_J(X)| & \le \sup_{x \in \cS(c) \cap \cU}|\hQ_J(x) - Q_J(x)|  \nonumber\\ 
& \le \sum_{j \le J, \,k = 0, 1} \sup_{x \in \cS(c) \cap \cU}|\log\hf_{jk}(\hx_j) - \log f_{jk}(x_j)| \nonumber \\
 & \le \sum_{j \le J, \,k = 0, 1}  \sup_{x \in \cS(c)}|\hf_{jk}(\hx_j) - f_{jk}(x_j)|\frac{1}{\inf_{x \in \cS(c) \cap \cU} \eta_{3jk}} \nonumber \\
& \le  \sum_{j \le J, \,k = 0, 1} \sup_{x \in \cS(c)}|\hf_{jk}(\hx_j) - f_{jk}(x_j)| \frac{2}{d_{jk}}, \nonumber \\
& = o(1) \label{eq:QdiffNonpar}
\end{align}
where the third inequality is by Taylor's theorem, $\eta_{3jk}$ is between $f_{jk}(x_{j})$ and $\hf_{jk}(\hx_j)$, the last inequality is due to \eqref{eq:djk} which holds for large enough $n$, and the equality is due to \eqref{eq:sumdjk}.  
We conclude that $P(S \cap \{\one\{\hQ_J(X) \ge 0 \} = Y\}) - P(S \cap \{\one\{Q_J(X) \ge 0 \} = Y\}) \tozero$ as $n \toinf$ by noting that $\hQ_J(X)$ converges to $Q_J(X)$ and thus has the same sign as $Q_J(X)$ as $n \toinf$. %, and $Q_J(X)$ either diverges to $\pm \infty$ or converges to a continuous random variable.
\end{proof}
% % % % % % % % % % % % % % % % % % % %

\subsection{Proof of Theorem~\ref{thm:nonparR}} \label{app:nonparEstR}
\begin{proof}
Note 
\begin{align}
\hE(Y|\xi_{j} = u) & = \frac{\sum_{k=0}^1\sum_{i=1}^{n_k} k K(\frac{u-\hxi_{ijk}}{h_j}) }{\sum_{k=0}^1\sum_{i=1}^{n_k} K(\frac{u-\hxi_{ijk}}{h_j}) } \nonumber\\ 
& = \frac{\sum_{i=1}^{n_1} K(\frac{u-\hxi_{ij1}}{h_j})}{\sum_{i=1}^{n_1} K(\frac{u-\hxi_{ij1}}{h_j}) + \sum_{i=1}^{n_0} K(\frac{u-\hxi_{ij0}}{h_j})} \nonumber\\
& = \frac{\hpi_1 \hf_{j1}(u)}{\hpi_1 \hf_{j1}(u) + \hpi_0 \hf_{j0}(u)}, \nonumber
\end{align}
where $\hf_{jk}$ are the kernel density estimators with bandwidth $h_j$. So 
\begin{align*}
\hQ_J^R(x) = & \sum\onetoJ \log \left(\frac{\hpi_0 \hE(Y|\xi_{j} = \hx_j) }{\hpi_1 [1 - \hE(Y|\xi_{j} = \hx_j)]}\right) \\
& = \sum\onetoJ \log \left(\frac{\hf_{j1}(\hx_j)}{\hf_{j0}(\hx_j)}\right).
\end{align*}
Observe that $\hQ_J^R$ has the same form as $\hQ_J$, so this result follows from   Theorem~\autoref{thm:nonpar}.
\end{proof}
% % % % % % % % % % % % % % % % % % % %

\subsection{An Auxiliary Lemma} \label{app:thmGen}
Assuming $X$ is Gaussian under $k = 0, 1$, whence the criterion function \eqref{eq:fr_trunc} becomes
\begin{equation} \label{eq:bayesGaussProof}
Q_J^G(x) = \half \sum_{j=1}^J \left[(\log\lambda_{j0} - \log\lambda_{j1}) - \left( \frac{1}{\lambda_{j1}}(x_j - \mu_{j})^2 - \frac{1}{\lambda_{j0}}x_j^2 \right) \right]  > 0.
\end{equation}

%$\xi_j$ has mean $\mu_{jk}$ and variance $\lambda_{jk}$ under group $k = 0, 1$, for $j = 1, 2, \dots.$ 

Let $\zeta_j = \xi_j / \sqrt{\lambda_{j0}}$. Then 
\begin{gather*}
\zeta_j \distas{\Pi_0} N(0, 1), \quad \zeta_j \distas{\Pi_1} N(m_j, r_j\inv), \text{ and } \\
Q_J^G(X) = \half\sum\onetoJ[\log r_j - r_j(\zeta_j-m_j)^2 + \zeta_j^2],
\end{gather*}
where $\distas{\Pi_k}$ means the distribution under group $k$. 
Under Gaussian assumptions,  our Bayes classifier is a special case of the quadratic discriminant (which is not Bayes in general because it uses two different sets of projections), whose perfect classification properties were discussed in \cite{hall:13:1} in the context of censored functional observations.

For the proof of Theorem 3 we use the following auxiliary result. 

\begin{Lemma} \label{lem:gen}
Assume the predictors come from Gaussian processes. If $\sumjinf m_j^2 < \infty$ and $\sumjinf (r_j-1)^2 < \infty$, then $Q_J^G(X)$ converges almost surely to a random variable as $J\toinf$, in which case perfect classification does not occur. Otherwise perfect classification occurs.
\end{Lemma}

 This lemma is similar to Theorem~1 of \cite{hall:13:1}, but uses more transparent conditions and a proof technique based on the optimality property of Bayes classifiers which will be reused in the proof of Theorem~\autoref{thm:nonparTrue}.
Lemma~\autoref{lem:gen} states perfect classification occurs if and only if there are sufficient differences between the two groups in the mean or covariance functions in the directions of tail eigenfunctions. This perfect classification phenomenon occurs in non-degenerate infinite dimensional case because we effectively have infinitely many projection scores $\xi_j$ for classification.

\begin{proof}
Case 1: Assume $\sumjinf(r_j-1)^2 = \infty$ and that there exists a subsequence $r_{j_l}$ of $r_j$ that goes to $\infty$ (resp. 0) as $l\toinf$. Take a subsequence $r_{j_l} \toinf, r_{j_l} > 1$ (resp. $r_{j_l} \tozero, r_{j_l} < 1$) for all $l \ge 1$. Denoting the summand $(\log\lambda_{j0} - \log\lambda_{j1}) - [ (\xi_j - \mu_{j1})^2/\lambda_{j1} - \xi_j^2/\lambda_{j0} ]$ of \eqref{eq:bayesGaussProof} as $S_j^G$, for any $j \le J$ the misclassification rate $P(\one\{Q_J^G(X) \ge 0\} \ne Y)$ is smaller than or equal to $P(\one\left\{ S_{j}^G\ge 0\right\} \ne Y)$,  since the former is the Bayes classifier using the first $J$ projections, which minimizes the misclassification error among the class. Thus the misclassification rate of $Q_J^G(X)$ is bounded above by that of the classifier $\one\{\log r_{j} - r_j(\zeta_j - m_j)^2 + \zeta_j^2 \ge 0 \}$ for any $j \le J$. Let $P_{\Pi_k}$ denote the conditional probability measure under group $k$. If there exists $r_{j_l} \tozero$, 
\begin{align*}
P_{\Pi_0}(\log r_{j_l} - r_{j_l}(\zeta_{j_l} - m_{j_l})^2 + \zeta_{j_l}^2 \ge 0) \le P_{\Pi_0}(\log r_{j_l} + \zeta_{j_l}^2 \ge 0) \tozero,
\end{align*}
observing $\zeta_{j_l}^2 \distas{\Pi_0} \chi_1^2$ and $\log r_{j_l} \rightarrow {-\infty}$. 

If there exist $r_{j_l} \toinf$,  then there exists a sequence $M\toinf$ such that $(\log r_{j_l} + M) / r_{j_l} \tozero$. For any $j \in \bbN$,
\begin{align}
P_{\Pi_0}(\log r_j - r_j(\zeta_j-m_j)^2 + \zeta_j^2 \ge 0) & \le P_{\Pi_0}(\log r_j - r_j(\zeta_j - m_j)^2 + M \ge 0)  + P(\zeta_j^2 > M)\nonumber\\
& = P_{\Pi_0}( (\zeta_j-m_j)^2 \le \frac{\log r_j + M}{r_j}) + o(1) \nonumber\\
& = P_{\Pi_0}( |\zeta_j-m_j| \le \sqrt{\frac{\log r_j + M}{r_j}}) + o(1) \label{eq:c1_2} 
%& \le P_{\Pi_0}\left(|\zeta_j| \le \sqrt{\frac{\log r_j + M}{r_j}}\right) + o(1). 
\end{align}

Plugging the sequence $r_{j_l}$ for $r_j$ into \eqref{eq:c1_2} we have $\sqrt{(\log r_j + M)/r_j} \tozero$ as $l\toinf$ and $M\toinf$. Since by (C1) the densities of $\zeta_j$ are uniformly bounded, \eqref{eq:c1_2} goes to zero and we have $P_{\Pi_0}(\log r_{j_l} - r_{j_l}(\zeta_{j_l}-m_{j_l})^2 + \zeta_{j_l}^2 \ge 0) \tozero$ as $l\toinf$ and $M\toinf$. Using similar arguments  we can also prove $P_{\Pi_1}(\log r_{j_l} - r_{j_l}(\zeta_{j_l}-m_{j_l})^2 + \zeta_{j_l}^2 < 0) \tozero$ as $l\toinf$. By Bayes theorem $P(\one\{S_{j_l}^G\ge 0 \} \ne Y) = P(Y=0)P(S_{j_l}^G\ge 0|Y=0) + P(Y=1)P(S_{j_l}^G < 0|Y=1) \tozero$ as $l\toinf$. 
Therefore 
\begin{align*}
P(\one\{Q_J^G(X) \ge 0\} \ne Y) \le P(\one\{S_{j_l}^G\ge 0 \} \ne Y) \tozero \text{ as } J \toinf,
\end{align*}
which means perfect classification occurs.

Case 2: Assume $\sumjinf(r_j-1)^2 = \infty$, and there exists $\lbarM$ and $\barM$ such that $0 < \lbarM \le r_j \le \barM < \infty $ for all $j\ge 1$. Letting $E_{\Pi_k}$ and $\Var_{\Pi_k}$ be the conditional expectation and variance under group $k$, respectively, we have
\begin{gather*}
E_{\Pi_0}[\log r_j - r_j(\zeta_j-m_j)^2 + \zeta_j^2 ] = \log r_j - (r_j - 1) - m_j^2r_j \\
E_{\Pi_1}[\log r_j - r_j(\zeta_j-m_j)^2 + \zeta_j^2] = -\log r_j\inv + (r_j\inv - 1 ) + m_j^2 \\
\Var_{\Pi_0}[\log r_j - r_j(\zeta_j-m_j)^2 + \zeta_j^2] = 2(1-r_j)^2 + 4m_j^2 r_j^2 \\
\Var_{\Pi_1}[\log r_j - r_j(\zeta_j-m_j)^2 + \zeta_j^2] = 2(r_j\inv-1)^2 + 4m_j^2r_j\inv.
\end{gather*}
Then 
\begin{align}
P_{\Pi_0}(\sum\onetoJ [\log r_j - r_j(\zeta_j-m_j)^2 + \zeta_j^2] \ge 0 ) & \le \frac{\sum\onetoJ [2(1-r_j)^2 + 4m_j^2 r_j^2] }{[-\sum\onetoJ (r_j - 1 -\log r_j + m_j^2r_j)]^2} \label{eq:m_j}\\
& \le \frac{\sum\onetoJ [2(1-r_j)^2 + 4\barM^2m_j^2] }{[\sum\onetoJ (\frac{1}{\barM} (r_j-1)^2 + \lbarM m_j^2)]^2} \nonumber\\
& = \frac{4\barM^2/\lbarM}{\sum\onetoJ [\frac{1}{\barM}(r_j-1)^2 + \lbarM m_j^2]} \cdot \frac{\sum\onetoJ [2(1-r_j)^2 + 4\barM^2m_j^2]}{\sum\onetoJ [4\frac{\barM}{\lbarM}(r_j-1)^2 + 4\barM^2m_j^2]} \nonumber\\
& \le \frac{4\barM^2/\lbarM}{\sum\onetoJ [\frac{1}{\barM}(r_j-1)^2 + \lbarM m_j^2]} \tozero \text{ as } J\tozero, \label{eq:chebyshev_gen}
\end{align}
where Chebyshev's inequality is used for the first inequality, and Taylor expansion in the second inequality. Analogously the misclassification rate under $\Pi_1$ also can be proven to go to zero. 

Case 3: Assume $\sumjinf (r_j-1)^2 < \infty$ and $\sumjinf m_j^2 = \infty$. The proof is essentially the same as in Case 2.

Case 4: Assume $\sumjinf(r_j-1)^2 < \infty$ and $\sumjinf m_j^2 < \infty$. Then the mean and variance of $Q_J^G(X)$ converges, so $Q_J^G(X)$ converges to a random variable under either population by \cite{bill:95}. Therefore misclassification does not occur.
\end{proof}

% % % % % % % % % % % % % % % % % % % %

\subsection{Proof of Theorem~\ref{thm:nonparTrue}} \label{app:nonpar}
\begin{proof}
%The proof technique is similar to that of Lemma~\autoref{lem:gen}.Lemma 1

Case 1: Assume $\sumjinf(r_j-1)^2 = \infty$ and there exists a subsequence $r_{j_l}$ of $r_j$ that goes to 0 or $\infty$ as $l\toinf$. By the optimality of Bayes classifiers, the Bayes classifier $\one\{Q_J(X) \ge 0 \}$ using the first $J$ components has smaller misclassification error than that of $\one\{S_{j} \ge 0 \}$, where $S_j$ is the $j$th component in the summand of \eqref{eq:fr_trunc}, for all $j \le J$. Since $\one\{S_{j} \ge 0 \}$ is the Bayes classifier using only the $j$th projection, it has a smaller misclassification error than the non-Bayes classifier $\one\{S_{j}^G \ge 0 \}$, where $S_j^G = \log r_j - r_j(\zeta_j - m_j)^2 + \zeta_j^2$ is the $j$th summand in \eqref{eq:bayesGaussProof}. With assumption (C1)-(C2), we prove the misclassification error goes to zeros by going through the same argument as in Lemma~\ref{lem:gen} Case 1.

Case 2: Assume $\sumjinf(r_j-1)^2 = \infty$, and there exists $\lbarM$ and $\barM$ such that $0 < \lbarM \le r_j \le \barM < \infty $ for all $j\ge 1$. By some algebra,
\begin{gather*}
E_{\Pi_0}[\log r_j - r_j(\zeta_j-m_j)^2 + \zeta_j^2 ] = \log r_j - (r_j - 1) - m_j^2r_j, \\
E_{\Pi_1}[\log r_j - r_j(\zeta_j-m_j)^2 + \zeta_j^2] = -\log r_j\inv + (r_j\inv - 1 ) + m_j^2, \\
\Var_{\Pi_0}[\log r_j - r_j(\zeta_j-m_j)^2 + \zeta_j^2] \le (2C_M - 1)(1-r_j)^2 + 4 (C_M + 1) m_j^2 r_j^2, \text{ and} \\
\Var_{\Pi_1}[\log r_j - r_j(\zeta_j-m_j)^2 + \zeta_j^2] \le (2C_M - 1)(r_j\inv-1)^2 + 4(C_M + 1)m_j^2r_j\inv.
\end{gather*}

The expectations are the same as in the Gaussian case because the first two moments of $\zeta_j$ does not depend on distributional assumptions. The inequalities in the variance calculation are due to $2ab \le a^2 + b^2 $ for all $a,b \in \bbR$. The same Chebyshev's inequality argument goes through as in Theorem~\ref{thm:nonpar}. 

Case 3: Assume $\sumjinf (r_j-1)^2 < \infty$ and $\sumjinf m_j^2 = \infty$. The proof is essentially the same as in Case 2.
\end{proof}

\subsection{Proof of Theorem~\ref{thm:presmooth}} \label{app:presmooth}

The proof requires the following key lemma, which is an extension of Lemma~\ref{lem:hff}, changing the rate from $h + (\frac{nh}{\log n})^{-\half}$ to $h + (\frac{nh}{\log n})^{-\half} + (m^{\frac{2}{5}}h^2)\inv$. The remainder of the proof is omitted, since it is analogous to that of Theorem~\autoref{thm:nonpar}.

\begin{Lemma} \label{lem:tff}
Assuming (B1)--(B4) and (D1)--(D4), for any $j \ge 1$, $k = 0, 1$,
\begin{equation} \label{eq:tff}
\sup_{x \in \cS(c)} |\tf_{jk}(\tx_j) - f_{jk}(x_j)| = O_P(h + (\frac{nh}{\log n})^{-\half} + (m^{\frac{2}{5}}h^2)\inv).
\end{equation}\end{Lemma}

%\subsection{Proof of Lemma~\ref{lem:tff}}
\begin{proof}
Given $x \in \cS(c)$, by triangle inequality
\[
|\tf_{jk}(\tx_j) - f_{jk}(x_j)| \le |\tf_{jk}(\tx_j) - \hf_{jk}(\hx_j)| + |\hf_{jk}(\hx_j) - f_{jk}(x_j)|.
\]

The rate for the second term can be derived from Lemma~\ref{lem:hff}, so we focus only on the first term. Note that for fixed $j,k$ and $h_{jk} = h\sqrt{\lambda_{jk}}$,
\begin{align}
|\tf_{jk}(\tx_j) - \hf_{jk}(\hx_j)| & = \frac{1}{n_k h_{jk}}\left| \sumink K\left(\frac{\int_\cT (\tXik(t) - x(t)) \tphi_j(t) \d t}{h_{jk}}\right) - K\left(\frac{\int_\cT (\Xik(t) - x(t)) \hphi_j(t) \d t}{h_{jk}}\right)\right| \nonumber \\
& \le \frac{1}{n_k h_{jk}^2} \sumink \left| \int_\cT (\tXik(t) - x(t)) \tphi_j(t) - (\Xik(t) - x(t)) \hphi_j(t) \d t \right| \cdot |K'(\eta_{4jk})| \nonumber \\
& \le \frac{c_3}{n_k h^2} \sumink \left| \int_\cT (\tXik(t) - x(t)) \tphi_j(t) - (\Xik(t) - x(t)) \hphi_j(t) \d t \right|, \label{eq:thdiff}
\end{align}
for a constant $c_3 > 0$, where the first inequality is by Taylor's theorem, $\eta_{4jk}$ is a mean value, and the last inequality is by (B4). The summand in \eqref{eq:thdiff} is %upper bounded by
\begin{align*}
& \left| \int_\cT (\tXik(t) - x(t)) \tphi_j(t) - (\Xik(t) - x(t)) \hphi_j(t) \d t \right| \\ 
= & \left| \int_\cT (\tXik(t) - \Xik(t)) \tphi_j(t) + (\Xik(t) - x(t))(\tphi_j(t) - \hphi_j(t)) \d t \right| \\
\le & \left| \int_\cT (\tXik(t) - \Xik(t)) \tphi_j(t)  \d t\right| + \left| \int_\cT (\Xik(t) - x(t))(\tphi_j(t) - \hphi_j(t)) \d t \right| \\
& \le ||\tXik - \Xik|| \cdot ||\tphi_j|| + ||\Xik - x|| \cdot || \tphi_j - \hphi_j|| \\
& \le ||\tXik - \Xik|| + (||\Xik|| + c) ||\tphi_j - \hphi_j||,
\end{align*}
where the second and third inequalities follow from Cauchy-Schwarz inequality and from $||x|| \le c$, respectively. Plugging the previous result into \eqref{eq:thdiff},
\begin{equation} \label{eq:thdiff2}
|\tf_{jk}(\tx_j) - \hf_{jk}(\hx_j)| \le \frac{c_3}{h^2} \left[ \oneovernk\sumink||\tXik - \Xik|| + ||\tphi_j - \hphi_j|| (\oneovernk\sumink ||\Xik|| + c) \right].
\end{equation}

Since $(\tXik, \Xik)$ are identically distributed for $i=1, \dots, n_k$, and that $\int_\cT E\{(\Xktwo(t)^2 \} \d t < \infty$ by (D1), the first term in the brackets has expected value equal to 
\begin{equation} 
E||\tXik - \Xik|| = E_{\Xik}[ E_{\varepsilon_i} ||\tXik - \Xik|| \;|\Xik ] = O((mw)^{-\half} + w^2) = O(m^{-2/5}), \nonumber
\end{equation}
where more details about the second equality in the last display can be found in the supplement of \cite{kong:16}. Also $E(\oneovernk\sumink ||\Xik|| + c) = O(1)$ by (B1). So 
\begin{equation} \label{eq:rateX}
\oneovernk\sumink||\tXik - \Xik||=O_p(m^{-2/5}) \quad \text{ and } \quad \oneovernk\sumink ||\Xik|| + c = O_p(1). 
\end{equation}

It remains to be shown $||\tphi_j - \hphi_j|| = O_p(m^{-2/5})$. 
Let $\tDelta_k = \tG_k - \hG_k$ and for a square-integrable function $A(s, t)$ denote $||A||_F = (\int_\cT \int_\cT A(s, t)^2 \d s \d t)^\half$ be the Frobenius norm.
\cite{kong:16} also shows in the supplement that $||\tDelta_k||_F = O_p(m^{-2/5})$,
%\begin{align*}
%||\tDelta_k||_F & = ||\frac{1}{n_k} \sumink(\tXik \otimes \tXik - \Xik \otimes \Xik) - (\bartXk \otimes \bartXk - \barXk \otimes \barXk)||_F \\
%& \le \oneovernk \sumink \left[||(\tXik - \Xik) \otimes (\tXik - \Xik)||_F + ||(\tXik - \Xik) \otimes \Xik||_F + ||\Xik \otimes (\tXik - \Xik)||_F\right] \\
%& \hphantom{=} + ||\bartXk \otimes \bartXk - \barXk \otimes \barXk||_F
%\end{align*}
%
%Note the first term in the last display is dominated by $\oneovernk \sumink||(\tXik - \Xik) \otimes \Xik||_F$. Since the sample trajectories are identical samples, we have
%\begin{align*}
%E(\oneovernk \sumink||(\tXik - \Xik) \otimes \Xik||_F) & = E(||(\tXik - \Xik) \otimes \Xik||_F)\\
%& = E(||\tXik - \Xik|| \cdot ||\Xik||) \\
%& \le E(||\tXik - \Xik||^2)^\half E(||\Xik||^2)^\half \\
%& = O((mw)^{-1/2} + w^2) O(1) = O(m^{-2/5}),
%\end{align*}
%in which we used Cauchy-Schwarz inequality. By expanding the products and carrying out similar calculations we obtain $E||\bartXk \otimes \bartXk - \barXk \otimes \barXk||_F = O(m^{-2/5})$. So $E||\tDelta_k|| = O(m^{-2/5})$ and thus $||\tDelta_k||_F = O_p(m^{-2/5})$. 
so $||\tDelta||_F = ||\tDelta_0 + \tDelta_1||_F / 2 = O_p(m^{-2/5})$. By standard perturbation theory for operators (see for example \cite{bosq:00}), for a fixed $j \ge 1$ 
\begin{equation} \label{eq:ratephi}
||\tphi_j - \hphi_j|| = O(||\tDelta||_F / \sup_{k\ne j} |\hlambda_j - \hlambda_k|) = O_p(m^{-2/5}).
\end{equation}
Plugging \eqref{eq:rateX} and \eqref{eq:ratephi} into \eqref{eq:thdiff2} % and by Lemma~\ref{lem:hff} 
we have the conclusion.
\end{proof}

\section*{References}
\addcontentsline{toc}{section}{References}
%\bibliography{../../../Papers/jabref_master,/Users/xiongtao/Dropbox/Papers/jabref_master}

\end{document}